\documentclass[12pt]{report}
\usepackage{epsfig}
\usepackage{amssymb,amsmath,amsthm,amscd}
\usepackage{latexsym}
\usepackage{fixltx2e}
\usepackage{mathrsfs}
\usepackage{verbatim}
\usepackage{url}
\usepackage{natbib}
\usepackage{nmtthesis2013}
\thesis
\author{Andrew Krause}
\title{Asymptotic Dynamics of Stochastic $p$-Laplace\\
  Equations on Unbounded Domains}
\degree{Master of Science in Mathematics\\
  with Specialty in Analysis}
\chair{Bixiang Wang}
\committeesize{3}

\newtheorem{theorem}{Theorem}[section]
\newtheorem{proposition}[theorem]{Proposition}

\newtheorem{lemma}[theorem]{Lemma}

\theoremstyle{definition}
\newtheorem{definition}[theorem]{Definition}

\newcounter{labelflag} 

\newcommand{\Label}[1]{
                       \ifnum\thelabelflag=1
                          \ifmmode
                             \makebox[0in][l]{\qquad\fbox{\rm#1}}
                          \else
                             \marginpar{\vspace{0.7\baselineskip}
                                        \hspace{-1.1\textwidth}
                                        \fbox{\rm#1}}
                          \fi
                       \fi
                       \label{#1} }

\newcommand{\be}{\begin{equation}}
\newcommand{\ee}{\end{equation}}

\newcommand{\abs}[1]{|#1|}
\newcommand{\norm}[1]{\|#1\|}

\newcommand{\R}{\mathbb{R}}
\newcommand{\N}{\mathbb{N}}

\def \calf {{  {\mathcal{F}} }}

\def \cala {{  {\mathcal{A}}_\alpha  }}

\def \calb {{  {\mathcal{B}}  }}
\def \cald {{  {\mathcal{D}}_\alpha }}

\def \zto {{  {z(\theta_t \omega)} }}

\def \zro {{ {z(\theta_r \omega)dr} }}
\def \pxk {{ {\rho(\frac{|x|^2}{k^2})} }}
\def \inp {{ {\int_{\mathbb{R}^n} \pxk} }}

 \def \o {{  \mathcal{O}  }}
 \def  \ltwo { {L^2 ({\R^n}) } }
  \def  \lq { { L^q  ({\R^n}) } }
   \def \ii {\int_{{\R^n}}}

 \def \eps{ { \varepsilon }}

\begin{document}
%
%
%
%
\titlepage
%
%
\begin{abstract}
This thesis is concerned with the asymptotic behavior of solutions of stochastic $p$-Laplace equations driven by non-autonomous forcing on $\R^n$. Two cases are studied, with additive and multiplicative noise respectively. Estimates on the tails of solutions  are used to overcome the non-compactness of Sobolev embeddings on unbounded domains, and prove asymptotic compactness of solution operators in $\ltwo$. Using this result we prove the existence and uniqueness of random attractors in each case. Additionally, we show the upper semicontinuity of the attractor for the multiplicative noise case as the intensity of the noise approaches zero.

\keywords{Random Attractors; Upper Semicontinuity; $p$-Laplace Equation}
\end{abstract}
\begin{acknowledgments}
I am indebted to the entire faculty of the New Mexico Tech for their teaching and support during my time there. I would like to especially thank my committee members, Dr. Ivan Avramidi and Dr. Bill Stone for countless hours of their time that I used to grow as a student and a person. I would also like to thank my wife, Sarah Krause, and my family for keeping my head from floating too high into the clouds these last few years. Finally, I would like to thank my advisor Dr. Bixiang Wang for his compassion and guidance in this work and his incredible support throughout my academic career. I am honored to have called Tech home for so long.
\end{acknowledgments}
\tableofcontents
%
%
\signaturepage
\chapter{Introduction}
Partial differential equations model a huge variety of phenomena in the physical, social, and life sciences. Unsurprisingly, solutions to these equations can often be just as varied and complex as the physical phenomenon being modeled. Therefore it becomes necessary to describe solutions to these equations, even when an analytic solution cannot be found. This qualitative endeavor has other benefits as well, such as being able to understand families of equations, and to predict behavior that numerical or approximate solutions would not readily find.

The subject of this thesis is the long-term behavior of solutions to two classes of stochastic degenerate parabolic equations with a $p$-Laplace term. Asymptotic behavior of solutions to these equations is investigated and used to establish the existence of random attractors. First suppose 
$(\Omega, \calf, P,  \{\theta_t\}_{t \in \R})$
 is a metric dynamical system where 
 $(\Omega, \calf, P)$ is a probability space
 and  $\{\theta_t\}_{t \in \R}$ is a measure-preserving
 transformation group on $\Omega$.
  Given $\tau \in\R$ and $\omega \in \Omega$, 
consider the   stochastic     equation 
defined for   $ x \in {\R^n}$ and   $t > \tau$,

\be
  \label{intr01}
  {\frac {\partial u}{\partial t}}  + \lambda u
  - {\rm div} \left (|\nabla u |^{p-2} \nabla u \right )
  =f(t,x,u ) 
  +g(t,x) 
  +\alpha \eta (\theta_t \omega) u   +\eps  h(x) {\frac {dW}{dt}},
\ee
with initial data,
 \be
 u( \tau, x ) = u_\tau (x),   \quad x\in  {\R^n}.
 \ee
The existence of attractors to this equation is studied in Chapter \ref{add_chap}. We also investigate existence and upper semicontinuity of attractors for the equation with multiplicative noise,

\be
  \label{intr02}
  {\frac {\partial u}{\partial t}}  + \lambda u
  - {\rm div} \left (|\nabla u |^{p-2} \nabla u \right )
  =f(t,x,u ) 
  +g(t,x) 
  +\alpha  u\circ {\frac {dW}{dt}},
 \ee
 in Chapter \ref{mult_chap}. In both cases we assume that $p\ge 2$,  $ \alpha>0  $,  $\lambda>0$,      $\eps>0 $,
  $f$ is a time dependent  nonlinearity, $g$  and $h$  are
  given functions,  
$\eta$ is a   random variable   and
$W$  is a    Wiener process on  $(\Omega, \calf, P)$.

Equations involving this type of nonlinear Laplacian operator have a rich mathematical theory, and play a role in some modern physical models. There are some good references on the mathematical aspects of this type of nonlinearity, from existence, uniqueness, and regularity of equations, to its role in a nonlinear generalization of potential theory. The reader may consult \cite{note_p} for further details. The physical models involving this type of an operator are those with nonlinear diffusion, such as in non-Newtonian fluids or glacial flow. We refer to \cite{lion1} for more details. 

It should also be noted that these two types of noise model different phenomena. Additive noise occurs in physical systems where forcing occurs due to factors that are not included in the model. A simple example of this type of noise would be wind in an earthquake model of a building. In general, additive noise is small in magnitude, and is not proportional to the unknown function being studied. In contrast, something like uncertainties in air resistance calculations might lead one to introduce multiplicative noise, where the uncertainty is dependent on the velocity of the object. There are many technical details here, both from the modeling and the mathematical standpoint. The reader is advised to consult standard books on stochastic calculus for further information, such as \cite{klebaner}.

An important observation in the study of dynamical systems is that analyzing long-term asymptotic behavior often reduces the possible dynamics of the system in question. This is frequently observed in dynamical systems which are dissipative in some sense. Many physical systems are modeled as dissipative dynamical systems due to friction and other thermodynamic losses. Studying these systems has given rise to various ways to model long-term asymptotic behavior. One particularly interesting approach is the idea of a global attractor which is an invariant subset of the phase space that attracts all trajectories.

The history of studying asymptotic dynamics is fairly old, and spans many areas of mathematics. Here we will restrict ourselves to looking at the recent developments, with particular emphasis on the theory of random attractors for partial differential equations.
Several good references exist for the classical (deterministic autonomous) theory, such as \cite{babin1, hale1, lady1, rob1, tem1}. Likewise, the more modern theory treating random and nonautonomous systems is vast. Some of the references that are particularly useful are \cite{carv1, cheb1, chu3, kloe2}.

 The
 definition of random attractor
 for autonomous stochastic systems
 was initially introduced in \cite{cra2, fla1,  schm1}.
 Since then,   
 such  attractors for autonomous stochastic PDEs
 have  been studied extensively, such as in 
  \cite{beyn1,  car1,car4, car5, 
   car3, car2, chu2, chu3, cra1, cra2, 
    fla1, gar2, gar1, gar3,  gess1, gess2, gess3,
      huang1, kloe1, schm1, shen1}  
       in bounded domains,
       and in     \cite{bat2, 
    wan1,  wan2}
         in  unbounded domains.
       For non-autonomous
       stochastic PDEs, the reader is referred to
         \cite{adi1, bat3,  car6, dua3, gess2, gess3,  kloe2,  wan5,
      wan7,  wan8}
     for   examples of the existence of  random   attractors.

If $f$  and $g$ do not depend on time,  then we
   call   \eqref{intr01}  an autonomous stochastic equation.
   In the autonomous  case,   the existence of   random attractors
   of \eqref{intr01}    has been  established recently
   in   \cite{gess2, gess3, gess1}
   by variational methods  under the condition that
    the growth rate
   of the nonlinearity $f$ is  not  bigger than $p$.
   This result   has been extended in \cite{wan8}
    to the case where    $f$ is non-autonomous and has a
    polynomial growth of any order. 
    Note that in all  papers mentioned above, 
    the  $p$-Laplace equation is defined in a bounded domain
    where compactness of Sobolev embeddings is available.
    Existence results on random
    attractors for   the stochastic $p$-Laplace equation defined on unbounded
    domains have been studied in \cite{krause1,krause2,li1}.  The goal of Chapter \ref{add_chap} is to
    overcome the non-compactness of Sobolev embeddings on $\R^n$
    and prove the existence and uniqueness of random
    attractors for \eqref{intr01} in $\ltwo$. 
    More precisely,   we will 
    show by a cut-off technique   that the tails of  solutions of \eqref{intr01} are uniformly
    small outside a bounded domain for large times.
    We then use this fact and 
      the compactness of solutions in bounded domains
      to establish the asymptotic compactness of solutions
      in $\ltwo$.
      By the asymptotic compactness and   absorbing sets of the equation,
      we  can obtain the existence and uniqueness  of random attractors.
      This  random  attractor    is   pathwise periodic if $f(t,x,u)$ and $g(t,x)$ are 
      periodic in $t$.

Similar approaches also yield existence results for a random attractor for equation \eqref{intr02}. Additionally, we demonstrate the convergence, in some sense, of the random attractor to the deterministic one as $\alpha \to 0$. These results can be found in Chapter \ref{mult_chap}, as well as in the journal article \cite{krause2}.


\chapter{Random Attractor Theory}
\setcounter{equation}{0}

 The following notation will be used throughout this thesis:
$\| \cdot \|$  for the norm  of $\ltwo$ and 
  $(\cdot, \cdot)$      for  its 
   inner product.
The norm of $L^p(\R^n)$ is usually  written  as $\| \cdot \|_p$
and the 
  norm  of a   Banach space  $X$
 is written as    $\|\cdot\|_{X}$.
 The symbol   $c$ or  $c_i$ ($i=1, 2, \ldots$) is used
 for a general positive 
      number  which may change from line to line.
 
 Finally, we recall the following inequality which will be used
 to interpolate between some spaces: 
\begin{equation}
\label{p_inequality}
\norm{u}^p_p \leq \frac{q-p}{q-2}\norm{u}^2 + \frac{p-2}{q-2}\norm{u}^q_q,
\ee
 where   $2 < p < q$  and  $u \in 
\ltwo \bigcap \lq$.  

The following discussion will be the basic theory necessary to discuss the asymptotic dynamics in the following chapters. We will specialize all results to the phase space $\ltwo$, but it should be noted that it is not difficult to extend these definitions and theorems to more general Banach spaces.

The following definition of a cocycle was introduced in \cite{wan5} in order to extend the notion of a cocycle for a random dynamical system to include systems that are also simultaneously driven by non-autonomous terms. Cocycles in general are extensions of the notion of semigroup or solution operator, in order to explicitly characterize the underlying dynamics of the probability space, and in this case the initial time. 

\begin{definition}\label{cocycle_definition} A map $\Phi:\R ^+ \times \R \times \Omega \times \ltwo
\to   \ltwo$ is called a Cocycle if
for all
  $\tau\in \R$,
  $\omega \in   \Omega $
  and    $t, s \in \R^+$,  
\begin{itemize}\sloppy
\item [(i)]   $\Phi (\cdot, \tau, \cdot, \cdot): \R ^+ \times \Omega \times \ltwo
\to   \ltwo$ is
 $(\calb (\R^+)   \times \calf \times \calb (\ltwo),  
\calb(\ltwo))$-measurable.

\item[(ii)]    $\Phi(0, \tau, \omega, \cdot) $ is the identity map on $ \ltwo$.

\item[(iii)]    $\Phi(t+ s, \tau, \omega, \cdot) =
 \Phi(t,  \tau +s,  \theta_{s} \omega, \cdot) 
 \circ \Phi(s, \tau, \omega, \cdot)$.

\item[(iv)]    $\Phi(t, \tau, \omega,  \cdot):  \ltwo \to   \ltwo$
 is continuous.
    \end{itemize}
		$\Phi$ is said to be $T$-periodic if for every $t \in \R^{+}$, $\tau \in \R$ and $\omega \in \Omega$,

		\begin{equation*}
			\Phi(t,\tau+T,\omega,\cdot) = \Phi(t,\tau,\omega,\cdot).
		\end{equation*}
\end{definition}
In order to discuss the notion of an attractor, we have to first discuss the domain of attraction. The idea of global attractor is a set which attracts all bounded subsets of the phase space. We can also discuss more general domains of attraction by defining a collection of nonempty bounded subsets of the phase space, which become time and path dependent random sets. This extension is in fact very useful to discuss uniqueness and compactness of attractors in certain cases. The reader is referred to the book \cite{carv1} for further discussion of these details. We define such a collection as,

\begin{equation*}
	\cald = \{D = \{D(\tau,\omega)\subseteq \ltwo : \tau \in \R, \omega \in \Omega \} \}.
\end{equation*}
Elements of $\cald$ must be bounded and nonempty, and the entire collection must also be inclusion-closed, which we define below.

\begin{definition}
A collection of sets $\cald$ is called inclusion-closed if whenever $D \in \cald$, and $D' \subseteq \ltwo$ is such that $D' \subseteq D$, then $D' \in \cald$.
\end{definition}
The set $\cald$ is often referred to as a universe in the literature. We will discuss the particular choice of universe for the attractors discussed later in the thesis. 

\begin{definition}\label{attractor_definition}A set $\cala  = \{\cala (\tau, \omega): \tau \in \R,
  \omega \in \Omega \} \in  \cald$ is called a $\cald$-random pullback attractor
  for  $\Phi$ in $\ltwo$ if the following are satisfied : 
 \begin{itemize}
\item [(i)]   $ {\mathcal{A}}$ is measurable
  and
 $ {\mathcal{A}}(\tau, \omega)$ is compact for all $\tau \in \R$
and    $\omega \in \Omega$.

\item[(ii)]   $ {\mathcal{A}}$  is invariant, that is,
for every $\tau \in \R$ and
 $\omega \in \Omega$,
$$ \Phi(t, \tau, \omega,  {\mathcal{A}}(\tau, \omega)   )
=  {\mathcal{A}} ( \tau +t, \theta_{t} \omega
), \ \  \forall \   t \ge 0.
$$

\item[(iii)]   For every
 $B = \{B(\tau, \omega): \tau \in \R, \omega \in \Omega\}
 \in  {\mathcal{D}}$ and for every $\tau \in \R$ and
 $\omega \in \Omega$,
$$ \lim_{t \to  \infty} \text{dist}_{\ltwo} (\Phi(t, \tau -t,
 \theta_{-t}\omega, B(\tau -t, 
 \theta_{-t}\omega) ) ,  {\mathcal{A}} (\tau, \omega ))=0,
$$
where  $\text{dist}_{\ltwo}$    is the Hausdorff   semi-distance
between two sets in $\ltwo$.
 \end{itemize}
\end{definition}
There are several important aspects of this definition. The $\cald$-pullback random attractor is parameterized by both initial time, $\tau$, and the path in the probability space, $\omega$. There is also an important dependence on the universe $\cald$ in which this object is attracting. We refer to the above cited literature for discussion and further motivation for why this particular object is interesting to study.

There are many different ways to prove the existence and uniqueness of an attractor, depending on the setting. Below we recall two definitions that will be used in this work, and are particularly important in the dynamics of PDEs, where compactness is a crucial technical concern.

\begin{definition}
			Let $K \in \cald$ be a family of nonempty closed subsets of $\ltwo$. Then $K$ is called a $\cald$-pullback absorbing set for $\Phi$ if for all $\tau \in \R, \omega \in \Omega$ and for every $B \in \cald$, there exists $T = T(B,\tau,\omega) > 0$ such that
			\begin{equation*}
				\Phi(t.\tau-t,\theta_{-t}\omega,B(\tau - t, \theta_{-t}\omega)) \subseteq K(\tau, \omega),\quad \text{for all} t \geq T. 
			\end{equation*}
			If, in addition, $K$ is measurable with respect to $\calf$ in $\Omega$, then $K$ is called a closed measurable $\cald$-pullback absorbing set of $\Phi$.
\end{definition}
\begin{definition}
			A Cocycle $\Phi$ on $\ltwo$ over $\R$ and $(\Omega, \calf, P, \theta)$ is called $\cald$-pullback asymptotically compact in $\ltwo$ if for every  
 $\tau \in \R$, $\omega \in \Omega$,
 $D  \in \cald$,  $t_n \to \infty$
 and $ u_{0,n}  \in D(\tau -t_n, \theta_{ -t_n} \omega )$,
 the sequence
 $\Phi (t_n, \tau -t_n,  \theta_{-t_n} \omega,   u_{0,n}  )$
  has a
   convergent
subsequence in $\ltwo $.
\end{definition}
Finally, we recall the following proposition that will be used throughout this work to guarantee existence, uniqueness, and periodicity of $\cald$-pullback random attractors. It is an extension of classical existence results to the definition of cocycle in the non-autonomous stochastic setting. The proof and further discussion can be found in \cite{wan5}.

\begin{proposition}\label{att}
 Let $ \cald$ be   the 
 collection given above.
If   
$\Phi$ is $ \cald$-pullback asymptotically
compact in $\ltwo$ and $\Phi$ has a  
closed  
   measurable  
     $ \cald$-pullback absorbing set
  $K$ in $ \cald$,  then
$\Phi$ has a unique  $\cald$-pullback
attractor $ \cala $  in $\ltwo$  which is  given  by, 
for each $\tau  \in \R$   and
$\omega \in \Omega$,
$$
 \cala (\tau, \omega)
=\Omega(K, \tau, \omega)
=\bigcup_{B \in  {\mathcal{D}}} \Omega(B, \tau, \omega),
$$
where $\Omega(K)
=\{ \Omega(K, \tau, \omega): \tau \in \R,
\omega \in \Omega\}$
is the $\omega$-limit set of $K$.

If, in addition, 
there is $T>0$  such
that  
$\Phi (t, \tau +T, \omega, \cdot)
=\Phi (t, \tau, \omega, \cdot)$
and $K(\tau +T, \omega) = K(\tau, \omega)$
for   all  $t\in \R^+$,  $\tau \in \R$
and $\omega \in \Omega$,
then the attractor $\cala$ is pathwise $T$-periodic, i.e.,
$\cala(\tau +T, \omega) = \cala(\tau, \omega)$.
  \end{proposition}

\chapter{Additive Noise}\label{add_chap}
\section{Cocycles  Associated with Degenerate Equations} 
\setcounter{equation}{0}

In this section, we  first establish the well-posedness of equation
\eqref{intr01} in $\ltwo$, and then define a continuous cocycle
for the   stochastic equation. This step is necessary for us
to investigate the asymptotic behavior of solutions.
 
Let   $(\Omega, \calf, P)$ be the standard probability space
where 
$ 
\Omega = \{ \omega   \in C(\R, \R ):  \omega(0) =  0 \}
$,   $\calf$   is 
 the Borel $\sigma$-algebra induced by the
compact-open topology of $\Omega$  and $P$
is   the   Wiener
measure on $(\Omega, \calf)$.   
Denote by  $\{\theta_t\}_{t \in \R}$ 
 the  family of shift operators given by 
$$
 \theta_t \omega (\cdot) = \omega (\cdot +t) - \omega (t)
\quad  \mbox{ for  all  } \   \omega \in \Omega \ 
\mbox{ and } \  t \in \R.
$$
From \cite{arn1} we know that 
 $(\Omega, \calf, P,  \{\theta_t\}_{t \in \R})$
 is a metric dynamical system.
    Given $\tau \in\R$ and $\omega \in \Omega$, 
consider the  following  stochastic     equation 
defined for   $ x \in {\R^n}$ and   $t > \tau$, 
\be
  \label{seq1}
  {\frac {\partial u}{\partial t}}  + \lambda u
  - {\rm div} \left (|\nabla u |^{p-2} \nabla u \right )
  =f(t,x,u ) 
  +g(t,x) 
  +\alpha \eta (\theta_t \omega) u   +\eps  h(x) {\frac {dW}{dt}}
 \ee  with  initial condition
 \be\label{seq2}
 u( \tau, x ) = u_\tau (x),   \quad x\in  {\R^n},
 \ee
 where $p\ge 2$,  $ \alpha>0  $,  $\lambda>0$,      $\eps>0 $,
$g \in L^2_{loc}(\R, \ltwo)$,  $h \in H^2({\R^n})$, 
$\eta$ is an integrable tempered  random variable   and
$W$  is a  two-sided real-valued Wiener process on  $(\Omega, \calf, P)$.
We  assume the nonlinearity 
$f: \R \times {\R^n} \times \R$ 
$\to \R$ is continuous and satisfies, 
for all
$t, s \in \R$   and  $x \in {\R^n}$, 
\be 
\label{f1}
f (t, x, s) s \le - \gamma
 |s|^q + \psi_1(t, x),  
\ee
\be 
\label{f2}
|f(t, x, s) |   \le \psi_2 (t,x)  |s|^{q-1} + \psi_3 (t, x),
\ee
\be 
\label{f3}
{\frac {\partial f}{\partial s}} (t, x, s)   \le \psi_4 (t,x),
\ee\sloppy
where $\gamma>0$  and $ q \ge p $
are constants,
$\psi_1 \in L^1_{loc} (\R,  L^1( {\R^n}) )$,
$\psi_2, \psi_4 \in L^\infty_{loc} (\R, L^\infty ( {\R^n}))$,
and $\psi_3 \in L^{q_1}_{loc} (\R, L^{q_1} (  {\R^n}))$.
From now on,   we  always  
assume  $h \in H^2({\R^n}) \bigcap  W^{2,q} ({\R^n})$
and   use $p_1$ and $q_1$
to denote 
  the conjugate exponents
of $p$ and $q$, respectively. 
Since  $h \in H^2({\R^n}) \bigcap  W^{2,q} ({\R^n})$
and $q \ge p$, by \eqref{p_inequality} we find
$h \in   W^{2,p} ({\R^n})$.

To define a random dynamical system for
  \eqref{seq1},  we need to transfer the stochastic equation
  to a pathwise  deterministic  system.
  As usual,  let $z$ be the 
  random variable  given by:
$$
z ( \omega)=   - \lambda  \int^0_{-\infty} e^{\lambda  \tau}    \omega  (\tau) d \tau,
\quad \omega \in \Omega.
$$
It  follows from \cite{arn1} that
 there exists a $\theta_t$-invariant 
 set $\widetilde{\Omega}  $  
 of full
 measure 
 such that  
  $z(\theta_t \omega)$  is
 continuous in $t$ 
and $ \lim\limits_{t \to \pm \infty}   {\frac { |z(\theta_t \omega)|}{|t|}}
= 0 $
for all  $\omega \in \widetilde{\Omega}$.
We also  assume    $\eta(\theta_t \omega)$  is
pathwise 
 continuous  for each fixed $\omega \in  \widetilde{\Omega}$.
For convenience,  we will  denote  
$\widetilde{\Omega}$    by   $\Omega$
in the sequel. 
  Let $u(t, \tau, \omega, u_\tau)$
   be a solution of problem \eqref{seq1}-\eqref{seq2}
   with initial condition
   $u_\tau$ at initial time $\tau$,  and  define
 \be
 \label{uv}
 v(t, \tau, \omega, v_\tau)
   =   u(t, \tau, \omega, u_\tau)
    - \eps  h(x) z(\theta_{t} \omega)
 \quad \mbox{with }   \
 v_\tau =   u_\tau - \eps  h  z(\theta_\tau \omega).
 \ee
 By \eqref{seq1} and \eqref{uv}, 
 after simple calculations,
  we get
 \be 
 \label{veq1}
\frac{\partial v}{\partial t}-\text{div}
\left (
\abs{\nabla
 (v+\eps h(x)z(\theta_{t} \omega))}^{p-2}
 \nabla (v+\eps h(x)z(\theta_{t} \omega))
 \right )
 + \lambda v 
 $$
 $$
 = f(t,x,v+\eps h(x)z(\theta_{t} \omega))
+ g(t,x) + \alpha\eta(\theta_t\omega)v
+ \alpha \eps  \eta(\theta_t\omega)     z(\theta_t\omega)  h,
\ee
with initial   condition
\be 
 \label{veq2}
v(\tau, x )=v_{\tau}(x ), \quad  x\in {\R^n}.
\ee
In what follows,  we first prove the well-posedness
of problem  \eqref{veq1}-\eqref{veq2}
in $\ltwo$, and then define
a cocycle   for \eqref{seq1}-\eqref{seq2}.

 \begin{definition}\sloppy
 \label{defv}
 Given $\tau \in \R$, $\omega \in \Omega$
 and
 $v_\tau \in \ltwo$, 
 let  $v(\cdot, \tau, \omega, v_\tau)$: $[\tau, \infty) \to \ltwo$
 be a continuous function with
  $v
 \in L^p_{loc}([\tau, \infty), W^{1,p}({\R^n}) ) \bigcap
 L^q_{loc} ([\tau, \infty), \lq)
 $  and 
 ${\frac {dv}{dt}}
 \in L^{p_1}_{loc}([\tau, \infty),  (W^{1,p}) ^*) 
 + 
 L^{2}_{loc} ( [\tau, \infty),  \ltwo )
 + 
 L^{q_1}_{loc} ( [\tau, \infty), L^{q_1} ({\R^n} ) )$.
 We say $v$ is a solution of   \eqref{veq1}-\eqref{veq2}
 if  $ v(\tau, \tau, \omega, v_\tau) =v_\tau$  
   and     for every\\
   $\xi \in   W^{1,p}({\R^n}) \bigcap \ltwo  \bigcap  \lq $,
 $$
 {\frac {d}{dt}}   (v, \xi)
 +\int_{\R^n} 
 \abs{\nabla
 (v+\eps h z(\theta_{t} \omega))}^{p-2}
 \nabla (v+\eps h z(\theta_{t} \omega) )
\cdot   \nabla \xi dx 
+ \left (\lambda - \alpha\eta(\theta_t\omega) \right )
 (v, \xi)
 $$
 $$
 =
 \int_{\R^n} f(t,x,v+\eps  h z(\theta_{t} \omega)) \xi dx
+ (g(t, \cdot), \xi) 
+ \alpha \eps  \eta(\theta_t\omega)     z(\theta_t\omega)  
(h, \xi)
   $$
 in the sense of distribution on $[\tau, \infty)$.
 \end{definition}
 
 Next, we prove the existence and uniqueness of solutions
 of \eqref{veq1}-\eqref{veq2} in $\ltwo$. To this end,
 we  set $\o_k = \{ x \in \R^n:  |x| < k \}$ for each $k \in \N$
 and consider the following equation defined in $\o_k$:
 \be 
 \label{veqk1}
\frac{\partial v_k}{\partial t}-\text{div}
\left (
\abs{\nabla
 (v_k+\eps h(x)z(\theta_{t} \omega))}^{p-2}
 \nabla (v_k+\eps h(x)z(\theta_{t} \omega))
 \right )
 + \lambda v_k 
 $$
 $$
 = f(t,x,v_k+\eps h(x)z(\theta_{t} \omega))
+ g(t,x) + \alpha\eta(\theta_t\omega)v_k
+ \alpha \eps  \eta(\theta_t\omega)     z(\theta_t\omega)  h,
\ee
with boundary condition
\be\label{veqk2}
v_k (t, x) = 0 \quad   \text{for all } t>\tau \text{ and }  |x| =k
\ee
and  initial   condition
\be 
 \label{veqk3}
v(\tau, x )=v_{\tau}(x )  \quad  \text{ for all }  x\in \o_k.
\ee
 By the arguments in \cite{wan8},  one can show that
 if  \eqref{f1}-\eqref{f3} are fulfilled,  then
 for every $\tau \in \R$
 and   $\omega \in \Omega$,   system  \eqref{veqk1}-\eqref{veqk3}
  has  a unique  solution
 $v_k(\cdot, \tau, \omega, v_\tau)  $
 in the sense of Definition \ref{defv} with $\R^n$  replaced by
 $\o_k$.   Moreover, 
 $v_k(t, \tau, \omega, v_\tau)$  is  $(\calf, \calb (L^2(\o_k)))$-measurable
 with respect to $\omega \in \Omega$. 
 We now investigate 
  the limiting behavior of $v_k$
 as $k \to \infty$.
  For convenience,  we   write  $V_k
  =W_0^{1,p} (\o_k)$    and $V= W^{1,p} (\R^n)$.
  Let   
 $A$: $ V_k  \to V_k^*  $ be  the operator  given by
\be\label{Av}
 (A(v_1),  v_2 )_{(V_k^*, V_k)}
 = \int_{\o_k} | \nabla v_1|^{p-2} \nabla v_1 \cdot \nabla v_2 dx,
 \quad  \mbox{for all} \ v_1, v_2 \in V_k,
\ee
 where $ (\cdot,   \cdot )_{(V_k^*, V_k)}$ is the duality pairing 
 of $V_k^*$  and $V_k$.
 Note that $A$  is a monotone operator
 as in \cite{show1} and  $A: V \to V^*$  is also well defined by
 replacing $\o_k$ by $\R^n$ in \eqref{Av}.
 The following uniform estimates on $v_k$  are useful.
  
 \begin{lemma}
 \label{lemvk}
 Suppose \eqref{f1}-\eqref{f3} hold. Then for every  $T>0$, 
 $\tau \in \R$, $\omega \in \Omega$  and $v_\tau \in \ltwo$,
 the solution $v_k(t, \tau, \omega, v_\tau)$ of system
 \eqref{veqk1}-\eqref{veqk3} has the properties:
$$
\{v_k\}_{k=1}^\infty
\ \mbox{ is bounded in } \
$$
$$
 L^\infty (\tau, \tau +T;  L^2(\o_k) ) \bigcap L^q(\tau, \tau +T;  L^q(\o_k) )
 \bigcap L^p(\tau, \tau +T; V_k),
$$
$$
 \{A( v_k+\eps h z(\theta_{t} \omega) )\}_{k=1}^\infty
\ \mbox{ is bounded in } \
  L^{p_1}(\tau, \tau +T; V_k^*)  \  \text{ with }  \ {\frac 1{p_1}} + {\frac 1p} =1,
  $$
  $$
  \{   f(t,x,v_k+\eps h z(\theta_{t} \omega))  \}_{k=1}^\infty
 \  \mbox{ is bounded in }  \
 L^{q_1} (\tau, \tau +T; L^{q_1} ({\o_k})),
 \quad  {\frac 1{q_1}} + {\frac 1q} =1,
$$
and
$$
  \left \{ {\frac {dv_k}{dt}}  \right \} 
  \mbox{ is bounded in }   
$$
$$
  L^{p_1}(\tau, \tau +T; V_k^*)   +  
    L^{2} (\tau, \tau +T; L^{2} ({\o_k}))
     +  
 L^{q_1} (\tau, \tau +T; L^{q_1} ({\o_k})).
 $$ \end{lemma}
  
  \begin{proof}
  By \eqref{veqk1}  we get
 \be
 \label{plvk1_1}
\frac{1}{2}\frac{d}{dt}\norm{v_k}^2 
+  \int_{\o_k} \abs{\nabla (v_k
+\eps h z(\theta_{t} \omega))}^{p-2}
\nabla (v_k+\eps h z(\theta_{t} \omega)  )
 \cdot  \nabla v_k dx
+ \lambda \norm{v_k}^2
$$
$$
= \int_{\o_k}
 f(t,x,v_k+\eps h z(\theta_{t} \omega) ) v_k dx
 + (g(t), v_k)
$$
$$
 + \alpha\eta(\theta_t\omega) \norm{v_k}^2
 +  \alpha\eps \eta(\theta_t\omega)  z(\theta_t\omega ) (h, v_k)   .
\ee
For the second  term on the left-hand side of \eqref{plvk1_1},  
by Young\rq{}s inequality we  obtain
$$ \int_{\o_k} \abs{\nabla (v_k
+\eps h z(\theta_{t} \omega))}^{p-2}
\nabla (v_k+\eps h z(\theta_{t} \omega)  ) \cdot  \nabla v_k dx
$$
$$
  = \int_{\o_k}\abs{\nabla (v_k +\eps h z(\theta_{t} \omega))}^{p}dx 
$$
$$
  - \int_{\o_k} \abs{\nabla (v_k+\eps h z(\theta_{t} \omega))}^{p-2}
  \nabla(v_k +\eps h z(\theta_{t}\omega) ) 
  \cdot \nabla  ( \eps h z(\theta_{t} \omega ) ) dx 
  $$
  \be\label{plvk1_2}
  \ge
  {\frac 12}  \int_{\o_k}\abs{\nabla (v_k +\eps h z(\theta_{t} \omega))}^{p}dx 
   -  c_1 \abs{\eps z(\theta_{t} \omega) }^p \| \nabla h \|^p_p.
  \ee
  For the first   term on the right-hand side of \eqref{plvk1_1},
 by  \eqref{f1} and \eqref{f2} we get
   \be
    \label{plvk1_3}
\int_{\o_k}
f(t,x,v_k+\eps h  z(\theta_{t} \omega)) v_k dx 
$$
$$
=  \int_{\o_k}f(t,x,v_k+\eps h z(\theta_{t} \omega))
(v_k +\eps h z(\theta_{t} \omega))dx 
$$
$$
- \eps z(\theta_{t} \omega) \int_{\o_k}
f(t,x,v_k +\eps h z(\theta_{t} \omega))h(x)dx
$$
$$
\leq -\gamma\int_{\o_k}
\abs{v_k+\eps h z(\theta_{t} \omega))}^q dx 
+ \int_{\o_k}\psi_1(t,x)dx
$$
$$
+
\int_{\o_k}\psi_2(t,x)
\abs{v_k+\eps h z(\theta_{t} \omega))}^{q-1}
\abs{\eps h z(\theta_{t} \omega)}dx 
+ \int_{\o_k}\psi_3(t,x) \abs{\eps h z(\theta_{t} \omega)}dx
$$
$$
\leq -\frac{\gamma}{2}\norm{v_k+\eps h z(\theta_{t} \omega))}^q_q 
+ \norm{\psi_1(t)}_{1}
+\norm{\psi_3(t)}_{q_1}^{q_1} 
+ c_2 \int_{\o_k}\abs{\eps h z(\theta_{t} \omega)}^q dx.
\ee
 By Young\rq{}s inequality   we obtain 
\be \label{plvk1_4}
\int_{\o_k} g(t,x) v_k dx  + \alpha \eps \eta(\theta_t\omega)z(\theta_t\omega)
\int_{\o_k} h(x) v_k dx  
$$
$$
\leq  
  \frac{4}{\lambda} \abs{ \alpha \eps \eta(\theta_t\omega)
z(\theta_t\omega)}^2 \norm{h} ^2
 + \frac{4}{\lambda}  \norm{g (t)}^2  + \frac{\lambda}{8}\norm{v_k}^2.
\ee
It follows   from \eqref{plvk1_1}-\eqref{plvk1_4}  that
  \be
  \label{plvk1_6}
\frac{d}{dt}\norm{v_k}^2 
+ {\frac{7}{4}}  \lambda \norm{v_k}^2 
+ \int_{\o_k} \abs{\nabla (v_k+\eps hz(\theta_{t} \omega))}^p_p  dx
+ \gamma
\int_{\o_k} \abs{v_k+\eps h z(\theta_{t} \omega))}^q_q dx
$$
$$
\leq 2\alpha\eta(\theta_t\omega)\norm{v_k}^2 
+ c_3 \left (  \abs{\eps z(\theta_{t} \omega)}^p 
+   \abs{\eps z(\theta_{t} \omega)}^q 
+     \abs{\alpha \eps \eta(\theta_t\omega)z(\theta_t\omega) }^2
\right )
$$
$$
 + c_4 \left (
 \norm{g(t)}^2\
 + \norm{\psi_1(t)}_{1}+   \norm{\psi_3(t)}_{q_1}^{q_1}
 \right ) .
\ee
 Multiplying \eqref{plvk1_6}
  by $e^{\frac{7}{4}\lambda t-2\alpha\int^{t}_{0}
  \eta(\theta_{r}\omega)dr}$, 
  and then  integrating from $\tau $ to $t$, we get 
  \be 
  \label{plvk1_8}
\norm{v_k(t, \tau,\omega,v_{\tau})}^2 
$$
$$
+ \int_{\tau}^{t}e^{\frac{7}{4}\lambda(s-t)
-2\alpha\int_{t}^{s}\eta(\theta_{r}\omega)dr}
\int_{\o_k} \abs{\nabla (v_k(s, \tau,\omega,v_{\tau})
+\eps h  z(\theta_{s} \omega))}^p_p dxds
$$
$$
+ \gamma\int_{\tau}^{t}e^{\frac{7}{4}\lambda(s-t)
-2\alpha\int_{t}^{s}\eta(\theta_{r}\omega)d r}
\int_{\o_k}\abs{v_k(s, \tau,\omega,v_{\tau})
+\eps h  z(\theta_{s} \omega))}^q_q dxds
$$
$$
\le  c_3  \int_{\tau}^{t}e^{\frac{7}{4}\lambda(s-t)-2\alpha\int_{t}^{s}
\eta(\theta_{r}\omega)d r}
\left (
\abs{\eps z(\theta_{s} \omega)}^p 
+  \abs{\eps z(\theta_{s} \omega)}^q 
+  \abs{\alpha \eps \eta(\theta_{s}\omega)z(\theta_s\omega)}^2
\right ) ds
$$
$$
+ c_4 \int_{\tau}^{t}e^{\frac{7}{4}
\lambda(s-t)-2\alpha\int_{t}^{s}\eta(\theta_{r}\omega)d r}
\left (
\norm{g(s)}^2
 + \norm{\psi_1(s)}_{1}+\norm{\psi_3(s)}_{q_1}^{q_1} 
\right )ds 
$$
$$
+ e^{\frac{7}{4}\lambda(\tau-t)-2\alpha\int_{t}^{\tau}
\eta(\theta_{r}\omega)d r}\norm{v_{\tau}}^2 .
\ee
By \eqref{plvk1_8} we get
\be\label{plvk1_10}
\{v_k\} 
\ \mbox{ is bounded in } \
$$
$$
 L^\infty (\tau, \tau +T;  L^2(\o_k) ) \bigcap L^q(\tau, \tau +T;  L^q(\o_k) )
 \bigcap L^p(\tau, \tau +T; V_k).
\ee
By \eqref{f2} and \eqref{plvk1_10} we obtain
\be\label{plvk1_11}
  \{   f(t,x,v_k+\eps h z(\theta_{t} \omega))  \}_{k=1}^\infty
 \  \mbox{ is bounded in }  \
 L^{q_1} (\tau, \tau +T; L^{q_1} ({\o_k}) ).
\ee
By \eqref{Av}  and \eqref{plvk1_10} we get
\be\label{plvk1_12}
 \{A( v_k+\eps h z(\theta_{t} \omega)   ) \}_{k=1}^\infty
\ \mbox{ is bounded in } \
  L^{p_1}(\tau, \tau +T; V_k^*) .
  \ee
   By \eqref{plvk1_10}-\eqref{plvk1_12}  it follows from 
   \eqref{veqk1} that 
$$
  \left \{ {\frac {dv_k}{dt}}  \right \} 
  \mbox{ is bounded in }   
  L^{p_1}(\tau, \tau +T; V_k^*)   
$$
$$ +
    L^{2} (\tau, \tau +T; L^{2} ({\o_k}))
     +  
 L^{q_1} (\tau, \tau +T; L^{q_1} ({\o_k})),
$$
which completes   the proof.
  \end{proof}

 The next lemma is concerned with the well-posedness
 of \eqref{veq1}-\eqref{veq2} in $\ltwo$.

 \begin{lemma}
 \label{exiv}
 Suppose \eqref{f1}-\eqref{f3} hold. Then for every 
 $\tau \in \R$, $\omega \in \Omega$  and $v_\tau \in \ltwo$,
 problem \eqref{veq1}-\eqref{veq2} has a unique solution
 $v(t, \tau, \omega, v_\tau)$ in the sense of Definition \ref{defv}.
 In addition,  $v(t, \tau, \omega, v_\tau)$
 is  ($\calf, \calb (\ltwo))$-measurable
   in $\omega$ and continuous in
  $v_\tau$ in $\ltwo$   and   satisfies 
 $$
 {\frac d{dt}} \norm{ v(t, \tau, \omega, v_\tau) }^2
 + 2 \left (
  \lambda - \alpha \eta  (\theta_{t} \omega)
 \right ) \norm{ v }^2
 +2   \norm{ \nabla
 (v+\eps h z(\theta_{t} \omega))}^p_p
 $$
 $$
 = 2 \eps  z(\theta_{t} \omega)  \ii 
\abs{\nabla
 (v+\eps h z(\theta_{t} \omega))}^{p-2}
 \nabla (v+\eps h z(\theta_{t} \omega))
\cdot \nabla h dx
$$
\be\label{enlem1}
+2  \ii  f(t,x,v+\eps h z(\theta_{t} \omega)) v dx
+ 2 ( g(t), v)   
+ 2 \alpha \eps  \eta(\theta_t\omega)     z(\theta_t\omega)  (h, v)
\ee
  for almost  all $t \ge \tau$.
  \end{lemma}
  
  \begin{proof} 
  Let $T>0$,  $t_0 \in [\tau, \tau +T]$ and $v_k
  (t, \tau, \omega, v_\tau)$ be the solution of 
  system \eqref{veqk1}-\eqref{veqk3} defined
  in $\o_k$.   Extend $v_k$
  to the entire space $\R^n$
  by setting $v_k =0$    on $\R^n \setminus \o_k$
  and denote this  extension still by
  $v_k$.
    By Lemma \ref{lemvk} we find that 
 there exist
 $\widetilde{v}\in \ltwo$,
 $v \in L^\infty (\tau, \tau +T; \ltwo)  
 \bigcap L^p(\tau, \tau +T; V) \bigcap L^q(\tau, \tau +T; \lq) $,
 $\chi_1
 \in  L^{q_1} (\tau, \tau +T;  L^{q_1}({\R^n}) ) $,
  $\chi_2
 \in  L^{p_1} (\tau, \tau +T;  V^* ) $
 such that, up to a subsequence,
 \be\label{pexiv_1}
 v_k \to v \ \mbox{weak-star  in } \ 
 L^\infty (\tau, \tau +T; \ltwo ),
 \ee
  \be\label{pexiv_2}
 v_k \to v \ \mbox{weakly  in } \ 
 L^p (\tau, \tau +T;  V )
 \ \mbox{ and } \  L^q(\tau, \tau +T; \lq),
 \ee
  \be\label{pexiv_3}
 A(  v_k+\eps h z(\theta_{t} \omega))   \to  \chi_2  \ \mbox{weakly  in } \ 
 L^{p_1} (\tau, \tau +T;  V^* ),
\ee  \be\label{pexiv_4}
  f(t,x,v_k +\eps h z(\theta_{t} \omega))   \to  \chi_1  \ \mbox{weakly  in } \ 
 L^{q_1} (\tau, \tau +T; L^{q_1} ({\R^n}) ),
 \ee
 and
 \be\label{pexiv_5}
 v_k(t_0,\tau, \omega, v_\tau)
 \to  {\widetilde{v}}
  \ \mbox{weakly  in } \ 
  \ \ltwo.
 \ee
 On the other hand, by 
 the  compactness
 of embedding $W^{1,p}(\o_k)
 \hookrightarrow L^2(\o_k)$ 
 and Lemma \ref{lemvk}, we can choose a further
 subsequence  (not relabeled)
  by a diagonal process such that
  for each $k_0 \in \N$,
   \be\label{pexiv_6}
 v_k \to v
 \ \mbox{ strongly  in } \ 
 L^2(\tau, \tau +T; L^2(\o_{k_0}) ).
 \ee
 By \eqref{veqk1}
 and \eqref{pexiv_1}-\eqref{pexiv_4}
 one can  show  that 
 for every $\xi \in V\bigcap \ltwo  \bigcap \lq$,
 $$
 {\frac {d}{dt}} (v, \xi)
 +    (\chi_2, \xi)_{(V^*, V)}
 +  (\lambda -\alpha \eta(\theta_t \omega) ) (v, \xi)
 $$
 \be\label{pexiv_8}
 =
   (\chi_1, \xi)_{(L^{q_1}, L^q)}
   +  (g(t), \xi) +
   \alpha \eps \eta(\theta_t \omega) z(\theta_t \omega)
   (h, \xi) 
\ee
 in the sense of distribution.
By \eqref{pexiv_8} we find 
\be\label{pexiv_9}
 {\frac {dv}{dt}} 
 =
-    \chi_2
 +    \chi_1 
 - (\lambda -\alpha \eta(\theta_t \omega) )v
 +g 
 +\alpha \eps \eta(\theta_t \omega) z(\theta_t \omega) h
\ee
in 
$L^{p_1}(\tau, \tau +T; V^*)
+  L^{q_1} (\tau, \tau +T; L^{q_1} ({\R^n}) )
+  L^{2} (\tau, \tau +T;  \ltwo  )$,
which along with the fact  
  $v \in  L^\infty(\tau, \tau +T;  \ltwo)
  \bigcap L^p(\tau, \tau +T; V)
\bigcap L^q(\tau, \tau +T; \lq)$
implies (see, e.g., 
 \cite{lion1})  that 
$v \in C ([\tau, \tau +T], \ltwo)$ and 
\be\label{pexiv_12}
{\frac 12} {\frac d{dt}} \| v \|^2
= ( {\frac {dv}{dt}},  v)_{( V^*+ L^{q_1} +L^2, 
V\bigcap L^q \bigcap L^2 )}
 \quad \mbox{for  almost all } \  
t \in (\tau, \tau +T).
\ee
By \eqref{pexiv_1}-\eqref{pexiv_6}, we can argue as
in \cite{wan8} to show that
\be\label{pexiv_20}
 \chi_2 = 
 A(  v+\eps h z(\theta_{t} \omega))  ,
 \quad
   \chi_1  = 
  f(t,x,v  +\eps h z(\theta_{t} \omega) ),
  \quad  v(\tau) =v_\tau \ 
  \text{ and } \   v(t_0) =  { \widetilde{v}}.
 \ee
 By   \eqref{pexiv_8} and \eqref{pexiv_20} 
 we  find that 
  $v$ is a solution of problem \eqref{veq1}-\eqref{veq2}
  in the sense of Definition \ref{defv}.
  On the other hand, 
  by   \eqref{pexiv_9} and \eqref{pexiv_20} 
  we see  that $v$ satisfies  energy equation
  \eqref{enlem1}.
 
 We next  prove   the uniqueness of solutions.
 Let  $v_1$ and $v_2$   be the solutions of \eqref{veq1}
 and    $ {\widetilde{v}} = v_1 -v_2$. Then we have
 $$
 {\frac {d{\widetilde{v}}}{dt}}
 +  
  A(v_1 + \eps h    z(\theta_t \omega) ) -A(v_2 +\eps h    z(\theta_t \omega) )  
 +\lambda {\widetilde{v}}
 $$
 $$
 =\alpha \eta (\theta_t \omega) {\widetilde{v}}
 +   f (t,x,  v_1 + \eps h    z(\theta_t \omega)  )
 -f (t,x,  v_2 + \eps h    z(\theta_t \omega)  ),
 $$
 which along with \eqref{f3} and the   monotonicity 
 of $A$ yields,   for all $t \in [\tau, \tau +T]$,
 $$ {\frac {d}{dt}}
 \norm{\widetilde{v}}  ^2
 \le 2 \alpha  \eta (\theta_t \omega) \norm {\widetilde{v}}  ^2
 +2  
 \int_{\R^n} \psi_4 (t,x) | {\widetilde{v}} |^2 dx
  \le c   \| {\widetilde{v}} \|^2
 $$
 for some   positive constant $c$ depending on $\tau,
 T$ and $\omega$.
 By Gronwall\rq{}s lemma we get,
 for all $t \in [\tau, \tau +T]$,  
 \be\label{pexiv_30}
 \norm{ v_1(t, \tau, \omega, v_{1,\tau} ) - v_2(t, \tau, \omega, v_{2,\tau}) }^2
 \le e^{c  (t-\tau)} \norm{v_{1,\tau} - v_{2,\tau}  }  ^2 .
 \ee
So  the uniqueness and
 continuity of 
 solutions in  initial data  follow  
 immediately.
 
 Note that   \eqref{pexiv_5},
 \eqref{pexiv_20} and the uniqueness of solutions
 imply  that
 the entire  sequence $v_k(t_0, \tau, \omega, v_\tau)
 \to v(t_0, \tau, \omega, v_\tau)$ weakly in $\ltwo$
 for every fixed $t_0 \in [\tau, \tau +T]$
 and     $\omega \in \Omega$.
 By the  measurability  of 
  $v_k (t, \tau, \omega, v_\tau)$ in $\omega$,
  we obtain  the  measurability  of 
  $v (t, \tau, \omega, v_\tau)$ directly. 
 \end{proof}
 
 The following result is useful when
  proving the asymptotic compactness of solutions.

  \begin{lemma}
 \label{comv1}
 Let   \eqref{f1}-\eqref{f3}  hold and
 $\{v_n\}_{n=1}^\infty$ be a bounded sequence
 in $\ltwo$.
 Then for every   
  $\tau  \in \R$,  $t >\tau$ and 
 $\omega \in \Omega$,
 there exist  $v_0 \in L^2(\tau, t;   \ltwo)$
 and  a subsequence 
 $\{v(\cdot, \tau, \omega, v_{n_m})\}_{m=1}^\infty$
  of
 $\{v(\cdot, \tau, \omega, v_n)\}_{n=1}^\infty$
 such that 
 $v(s, \tau, \omega, v_{n_m})
 \to v_0 (s) $ in $L^2(\o_k)$
 as $m \to \infty$
 for every  fixed $k \in \N$
 and for almost all $s \in (\tau, t)$.
 \end{lemma}
 
 \begin{proof}
 Let $T$ be  a sufficiently large  number 
     such that
 $t \in (\tau, \tau +T]$. 
 Following the proof of \eqref{pexiv_6},  we 
 can show   that
    there  exists ${\widetilde{v}} \in L^2(\tau,  \tau +T;  \ltwo)$
 such   that,  up to  a subsequence,
 $$
  v(\cdot, \tau, \omega, v_{n} ) \to   {\widetilde{v}}
  \ \mbox{  strongly   in } \ L^2( \tau,  \tau+T;   L^2(\o_k ) )
  \quad \text{ for every }  k \in \N.
 $$
 Thus, for $k=1$,  there exist  a set $I_1 \subseteq [\tau, \tau +T]$
 of measure zero and a subsequence  
 $  v(\cdot, \tau, \omega, v_{n_1} )$ such that
 $$
  v(s, \tau, \omega, v_{n_1} ) \to   {\widetilde{v}} (s)
  \ \mbox{    in } \   L^2(\o_1)   
  \quad \mbox{    for  all }
 \    s \in [\tau,  \tau +T]\setminus I_1 .
 $$
  Similarly,  for $k=2$,  there
   exist  a set $I_2 \subseteq [\tau, \tau +T]$
 of measure zero and a subsequence  
 $  v(\cdot, \tau, \omega, v_{n_2} )$ of
 $  v(\cdot, \tau, \omega, v_{n_1} )$ 
   such that
$$
  v(s, \tau, \omega, v_{n_2} ) \to   {\widetilde{v}} (s)
  \ \mbox{    in } \   L^2(\o_2)   
  \quad \mbox{    for  all }
 \    s \in [\tau,  \tau +T]\setminus I_2 .
  $$
  Repeating this process we find that
  for each $k \in \N$, 
  there
   exist  a set $I_k \subseteq [\tau, \tau +T]$
 of measure zero and a subsequence  
 $  v(\cdot, \tau, \omega, v_{n_k} )$ of
 $  v(\cdot, \tau, \omega, v_{n_{k-1}} )$ 
   such that
 $$
  v(s, \tau, \omega, v_{n_k} ) \to   {\widetilde{v}}  (s)
  \ \mbox{    in } \   L^2(\o_k)   
  \quad \mbox{    for  all }
 \    s \in [\tau,  \tau +T]\setminus I_k .
 $$
  Let $I =\bigcup_{k=1}^\infty I_k$. Then
  by a diagonal process,  we infer that
  there exists a subsequence (which is
  still denoted by    
 $  v(\cdot, \tau, \omega, v_{n} )$)
   such that
 \be\label{pcomv1_1}
  v(s, \tau, \omega, v_{n} ) \to  {\widetilde{v}}  (s)
  \ \mbox{    in } \   L^2(\o_k)   
  \quad \mbox{    for  all }
 \    s \in [\tau,  \tau +T]\setminus I 
 \text{ and }   k   \in \N .
  \ee
  Note that $I$ has measure zero
  and   $t\in (\tau, \tau +T]$, which along with 
   \eqref{pcomv1_1}  completes  the proof.
 \end{proof}

  Based on Lemma \ref{exiv},  we  can
  define a continuous cocycle for
  problem  \eqref{seq1}-\eqref{seq2}
  in $\ltwo$.    
Let  $\Phi: \R^+ \times \R \times \Omega \times \ltwo$
$\to \ltwo$ be  a mapping  given by,
for every  $t \in \R^+$,  $\tau \in 
\R$, $\omega \in \Omega$  and $u_\tau \in \ltwo$,
 \be \label{pcycle}
 \Phi (t, \tau,  \omega, u_\tau)  
 =
  v(t+\tau, \tau,  \theta_{ -\tau} \omega,  v_\tau)
  + \eps h(x)   z(\theta_t \omega),
\ee 
where  $v$  is the solution of
system \eqref{veq1}-\eqref{veq2}
with initial condition $v_\tau =    u_\tau - \eps h(x) z(\omega) $
at initial time $\tau$. 
Note that \eqref{uv}  and \eqref{pcycle}
imply 
\be\label{pcycle2}
\Phi (t, \tau,  \omega, u_\tau)  
 =
  u(t+\tau, \tau,  \theta_{ -\tau} \omega,  u_\tau),
  \ee
  where $u$ is a solution of \eqref{seq1}-\eqref{seq2}
  in  some sense.
  Since   the solution $v$ of \eqref{veq1}-\eqref{veq2}
  is $(\calf, \calb(\ltwo))$-measurable in $\omega$
  and continuous in initial data in $\ltwo$, we find that
  $\Phi (t, \tau,  \omega, u_\tau)  $ given by \eqref{pcycle} 
  is also $(\calf, \calb(\ltwo))$-measurable in $\omega$
  and continuous in $u_\tau$ in $\ltwo$.
  In fact,  one can verify  that $\Phi$
  is a continuous cocycle  on $\ltwo$
over 
$(\Omega, \calf, P,  \{\theta\}_{t \in \R})$
in the sense of Definition \eqref{cocycle_definition}.
    Note that   the cocycle property (iii) of $\Phi$ can be easily proved
    by \eqref{pcycle} and the properties of the solution $v$
    of the pathwise deterministic equation \eqref{veq1}-\eqref{veq2}.
    Our goal is to   establish
     the existence of random    attractors  of $\Phi$
     with 
   an appropriate  attraction domain.
   To specify such an attraction domain, we  consider
    a family  
  $D =\{ D(\tau, \omega) \subseteq \ltwo: \tau \in \R, \omega \in \Omega \}$ 
  of   
  bounded nonempty    sets  such that
  for every $\tau \in \R$   and $\omega \in \Omega$, 
 \be
 \label{Dom1}
 \lim_{s \to  - \infty} e^{ {\frac 54} \lambda s
 + 2\alpha  \int^0_s   \eta (\theta_r \omega) dr }
  \| D( \tau + s, \theta_{s} \omega ) \|^2 =0,
\ee   where 
   $\| S \|=  \sup\limits_{u \in S}
   \| u\|_{\ltwo }$ for a nonempty bounded  subset $S$ of $\ltwo$.
  In  the sequel,   we will   use  $\cald$
   to denote the collection of all families 
   with property \eqref{Dom1}:
 \be
 \label{Dom2}
\cald  = \{ 
   D =\{ D(\tau, \omega) \subseteq \ltwo:
    \tau \in \R, \omega \in \Omega \}: \ 
 D  \ \mbox{satisfies} \  \eqref{Dom1} \} .
\ee
We   will construct a $\cald$-pullback attractor
$\cala  = \{\cala (\tau, \omega): \tau \in \R,
  \omega \in \Omega \} \in  \cald$ 
  for  $\Phi$ in $\ltwo$ in the sense of Definition \eqref{attractor_definition}.
 
 We  will apply  the result of Proposition \eqref{att} from 
   \cite{wan5} to show  the existence of 
   $\cald$-pullback attractors   for $\Phi$.
  Similar results on existence of random attractors
    can be found in 
\cite{bat1, car6,  cra2, fla1,  gess2, schm1}. 
   
We remark that  the $\calf$-measurability of the
   attractor   $ {\mathcal{A}}$  was  given   in  
\cite{wan7} and the  measurability of $ {\mathcal{A}}$
   with respect to 
   the $P$-completion
 of $\calf$   was given  in \cite{wan5}.
For our purpose,   we 
further assume  the following condition
on $g$, $\psi_1$  and $\psi_3$:
for every $ \tau \in \R$, 
 \be\label{g1}
\int_{-\infty}^\tau e^{\lambda s}
\left (
\| g(s, \cdot) \|^2 + \| \psi_1 (s, \cdot) \|_{L^1({\R^n})}
+ \| \psi_3 (s, \cdot) \|_{L^{q_1}({\R^n})}^{q_1}
\right ) ds < \infty.
\ee

\section{Uniform Estimates of Solutions}
 \setcounter{equation}{0}
 
 This section is devoted to 
  uniform estimates of  solutions   of 
  \eqref{seq1}  and \eqref{veq1}
  which are needed for proving the
  existence of random attractors for $\Phi$.
  When deriving uniform estimates,  
  the following     positive number
  $\alpha_0$ is useful:
 \be
 \label{alphazero}
\alpha_0 = \frac{1}{8(1+\abs{E(\eta)})}\lambda.
\ee

\begin{lemma}
\label{est1}
Let $\alpha_0$ be the  positive number given by \eqref{alphazero}.
 Suppose  \eqref{f1}-\eqref{f3}  and \eqref{g1} hold.
Then for every  $\alpha \le \alpha_0$,
 $\sigma \in \R$,
 $\tau \in \R$, $\omega \in \Omega$   and $D=\{D(\tau, \omega)
: \tau \in \R,  \omega \in \Omega\}  \in \cald$,
 there exists  $T=T(\tau, \omega,  D, \sigma, \alpha)>0$ such that 
 for all $t \ge T$,  the solution
 $v$ of  problem  \eqref{veq1}-\eqref{veq2}     satisfies 
 $$
\norm{v(\sigma, \tau-t,\theta_{-\tau}\omega,v_{\tau-t})}^2 
$$
$$
+
\int_{\tau-t}^{\sigma}e^{\frac{5}{4}\lambda(s-\sigma)
-2\alpha\int_{\sigma-\tau}^{s-\tau}
\eta(\theta_{r}\omega)dr}
\norm{v(s, \tau-t,\theta_{-\tau}\omega,v_{\tau-t})}^2ds 
$$
$$
+
\int_{\tau-t}^{\sigma}e^{\frac{5}{4}
\lambda(s-\sigma)-2\alpha\int_{\sigma-\tau}
^{s-\tau}\eta(\theta_{r}\omega)dr}
$$
$$
\times\norm{\nabla (v(s, \tau-t,\theta_{-\tau}\omega,v_{\tau-t})
+\eps h(x)z(\theta_{s-\tau} \omega))}^p_pds 
$$
$$
+
\int_{\tau-t}^{\sigma}e^{\frac{5}{4}
\lambda(s-\sigma)-2\alpha\int_{\sigma-\tau}
^{s-\tau}\eta(\theta_{r}\omega)dr}
$$
$$
\times\norm{v(s, \tau-t,\theta_{-\tau}\omega,
v_{\tau-t})+\eps h(x)z(\theta_{s-\tau} \omega))}^q_qds
\le M,
$$
where $v_{\tau -t}  \in D(\tau -t, \theta_{ -t} \omega)$
  and
  $M=M(\tau, \omega, \sigma, \alpha, \eps) $ is given by
$$
M
= c\int_{-\infty}^{\sigma-\tau}e^{\frac{5}{4}
\lambda(s-\sigma+\tau)-2\alpha\int_{\sigma-\tau}^{s}
\eta(\theta_{r}\omega)dr} 
$$
$$
\times \left ( \abs{\eps z(\theta_{s} \omega)}^p 
+ \abs{\eps z(\theta_{s} \omega)}^q 
+ \abs{\alpha \eps \eta(\theta_{s}\omega)z(\theta_{s}\omega )
}^2 \right )ds
$$
$$
+ c\int_{-\infty}^{\sigma-\tau}e^{\frac{5}{4}
\lambda(s-\sigma+\tau)-2\alpha
\int_{\sigma-\tau}^{s}\eta(\theta_{r}\omega)dr}
$$
$$
\times(\norm{g(s+\tau)}^2 
+ \norm{\psi_1(s+\tau)}_{1}
+\norm{\psi_3(s+\tau)}_{q_1}^{q_1})ds,
$$
with 
  $c$  being  a  positive constant  
   independent of $\tau$, $\omega$,  $D$,
  $\alpha$ and $\eps$.
\end{lemma}

\begin{proof}  
Using energy equation \eqref{enlem1}
and 
following the proof of \eqref{plvk1_6},
 we  obtain
  \be
  \label{pest1_1}
\frac{d}{dt}\norm{v}^2 
+ {\frac{7}{4}}  \lambda \norm{v}^2 
+ \ii \abs{\nabla (v+\eps hz(\theta_{t} \omega))}^p_p  dx
+ \gamma
\ii \abs{v+\eps h z(\theta_{t} \omega))}^q_q dx
$$
$$
\leq 2\alpha\eta(\theta_t\omega)\norm{v}^2 
+ c_3 \left (  \abs{\eps z(\theta_{t} \omega)}^p 
+   \abs{\eps z(\theta_{t} \omega)}^q 
+     \abs{\alpha \eps \eta(\theta_t\omega)z(\theta_t\omega) }^2
\right )
$$
$$
 + c_4 \left (
 \norm{g(t)}^2\
 + \norm{\psi_1(t)}_{1}+   \norm{\psi_3(t)}_{q_1}^{q_1}
 \right ) .
\ee
Multiplying \eqref{pest1_1} 
by $e^{\frac{5}{4}\lambda t-2\alpha\int^{t}_{0}
\eta(\theta_{r}\omega)dr}$, and
then  integrating from $\tau -t$ to $\sigma$
with $\sigma \ge \tau -t$, 
we get,  
 \be \label{pest1_2}
\norm{v(\sigma, \tau-t,\omega,v_{\tau-t})}^2
 + \frac{\lambda}{2} \int_{\tau-t}^{\sigma}
 e^{\frac{5}{4}\lambda(s-\sigma)
 -2\alpha\int_{\sigma}^{s}
 \eta(\theta_{r}\omega)dr}
 \norm{v(s, \tau-t,\omega,v_{\tau-t})}^2ds
 $$
 $$
+ \int_{\tau-t}^{\sigma}e^{\frac{5}{4}
\lambda(s-\sigma)-2\alpha\int_{\sigma}^{s}
\eta(\theta_{r}\omega)dr}
\norm{\nabla (v(s, \tau-t,\omega,v_{\tau-t})
+\eps hz(\theta_{s} \omega))}^p_pds
$$
$$
+ \gamma\int_{\tau-t}^{\sigma}
e^{\frac{5}{4}\lambda(s-\sigma)
-2\alpha\int_{\sigma}^{s}\eta(\theta_{r}\omega)dr}
\norm{v(s, \tau-t,\omega,v_{\tau-t})
+\eps h z(\theta_{s} \omega) }^q_qds
$$
$$
\leq e^{\frac{5}{4}\lambda(\tau-t-\sigma)
-2\alpha\int_{\sigma}^{\tau-t}\eta(\theta_{r}\omega)dr}
\norm{v_{\tau-t}}^2  
$$
$$
+ c_3 \int_{\tau-t}^{\sigma}e^{\frac{5}{4}
\lambda(s-\sigma)-2\alpha\int_{\sigma}^{s}
\eta(\theta_{r}\omega)dr}( \abs{\eps z(\theta_{s} \omega)}^p
 +  \abs{\eps z(\theta_{s} \omega)}^q 
 +  \abs{\alpha \eps \eta(\theta_{s}\omega)z(\theta_s\omega)}^2
 )ds
 $$
 $$
+ c_4 \int_{\tau-t}^{\sigma}
e^{\frac{5}{4}\lambda(s-\sigma)
-2\alpha\int_{\sigma}^{s}\eta(\theta_{r}\omega)dr}
 \left (
 \norm{g(s)}^2
 + 
 \norm{\psi_1(s)}_{1}+\norm{\psi_3(s)}_{q_1}^{q_1} 
 \right )
 ds.
\ee
 Replacing $\omega$ with $\theta_{-\tau}\omega$
  in \eqref{pest1_2},  we get
 \be
 \label{pest1_4}
\norm{v(\sigma, \tau-t,\theta_{-\tau}\omega,v_{\tau-t})}^2 
$$
$$
+ \frac{\lambda}{2} \int_{\tau-t}^{\sigma}
e^{\frac{5}{4}\lambda(s-\sigma)
- 2\alpha\int_{\sigma-\tau}^{s-\tau}\eta(\theta_{r}\omega)dr}
\norm{v(s, \tau-t,\theta_{-\tau}\omega,v_{\tau-t})}^2ds
$$
$$
+ \int_{\tau-t}^{\sigma}e^{\frac{5}{4}\lambda(s-\sigma)
-2\alpha\int_{\sigma-\tau}^{s-\tau}\eta(\theta_{r}\omega)dr}
\norm{\nabla (v(s, \tau-t,\theta_{-\tau}\omega,v_{\tau-t})
+\eps h z(\theta_{s-\tau} \omega))}^p_pds
$$
$$
+ \gamma\int_{\tau-t}^{\sigma}e^{\frac{5}{4}
\lambda(s-\sigma)-2\alpha\int_{\sigma-\tau}^{s-\tau}
\eta(\theta_{r}\omega)dr}
\norm{v(s, \tau-t,\theta_{-\tau}\omega,v_{\tau-t})
+\eps h z(\theta_{s-\tau} \omega)}^q_qds
$$
$$
\leq e^{\frac{5}{4}\lambda
 (\tau -t -\sigma)+2\alpha\int^{\sigma -\tau}_{-t}
\eta(\theta_{r}\omega)dr}  \norm{v_{\tau-t}}^2
$$
$$
+ c_3 \int_{-t}^{\sigma-\tau}e^{\frac{5}{4}
\lambda(s +\tau -\sigma)-2\alpha\int_{\sigma -\tau}^{s}
\eta(\theta_{r}\omega)dr}
$$
$$
\times( \abs{\eps z(\theta_{s} \omega)}^p
 +  \abs{\eps z(\theta_{s} \omega)}^q 
 +  \abs{\alpha \eps \eta(\theta_{s}\omega)z(\theta_s\omega)}^2
 )ds
 $$
 $$
+ c_4 
 \int_{-t}^{\sigma-\tau}e^{\frac{5}{4}
\lambda(s +\tau -\sigma)-2\alpha\int_{\sigma -\tau}^{s}
\eta(\theta_{r}\omega)dr}
$$
$$
\times \left (
 \norm{g(s+\tau)}^2
 + 
 \norm{\psi_1(s+\tau)}_{1}+\norm{\psi_3(s+\tau)}_{q_1}^{q_1} 
 \right )
 ds.
 \ee
 By the ergodicity of $\eta$, \eqref{alphazero}  and \eqref{g1}
 one can verify  that   for all $\alpha \le \alpha_0$, 
 \be\label{pest1_7}
  \int_{-\infty}^{\sigma-\tau}e^{\frac{5}{4}
\lambda(s +\tau -\sigma)-2\alpha\int_{\sigma -\tau}^{s}
\eta(\theta_{r}\omega)dr}
$$
$$
\times \left (
 \norm{g(s+\tau)}^2
 + 
 \norm{\psi_1(s+\tau)}_{1}+\norm{\psi_3(s+\tau)}_{q_1}^{q_1} 
 \right )
 ds <\infty.
 \ee
 Similarly, by the temperedness of $\eta$  and  $z$, we can prove  that
 for all  $\alpha \le \alpha_0$,
 \be\label{pest1_9}
 \int_{-\infty}^{\sigma-\tau}e^{\frac{5}{4}
\lambda(s +\tau -\sigma)-2\alpha\int_{\sigma -\tau}^{s}
\eta(\theta_{r}\omega)dr}
$$
$$
\times( \abs{\eps z(\theta_{s} \omega)}^p
 +  \abs{\eps z(\theta_{s} \omega)}^q 
 +  \abs{\alpha \eps \eta(\theta_{s}\omega)z(\theta_s\omega)}^2
 )ds <\infty.
 \ee
 Since $v_{\tau -t} \in    D(\tau-t, \theta_{-t}\omega)$ 
 and $D \in \cald$, by \eqref{Dom1}-\eqref{Dom2} we obtain
 $$
 e^{\frac{5}{4}\lambda
 (\tau -t -\sigma)+2\alpha\int^{\sigma -\tau}_{-t}
\eta(\theta_{r}\omega)dr}  \norm{v_{\tau-t}}^2
 $$
  $$
 \le e^{\frac{5}{4}\lambda
 (\tau  -\sigma)+2\alpha\int^{\sigma -\tau}_{0}\eta(\theta_{r}\omega)dr}
 e^{- \frac{5}{4}\lambda t 
 +2\alpha\int^{0}_{-t}
\eta(\theta_{r}\omega)dr}
  \norm{D(\tau -t, \theta_{-t} \omega)}^2
  \to 0,
 $$
 as $t \to \infty$.
 Therefore, 
  there exists $T = T(\tau, \omega,  D, \sigma, \alpha)>0$ 
  such that for all $t \geq T$, 
  $$
 e^{\frac{5}{4}\lambda
 (\tau -t -\sigma)+2\alpha\int^{\sigma -\tau}_{-t}
\eta(\theta_{r}\omega)dr}  \norm{v_{\tau-t}}^2
$$
$$
\le
 \int_{-\infty}^{\sigma-\tau}e^{\frac{5}{4}
\lambda(s +\tau -\sigma)-2\alpha\int_{\sigma -\tau}^{s}
\eta(\theta_{r}\omega)dr}
$$
$$
\times \left (
 \norm{g(s+\tau)}^2
 + 
 \norm{\psi_1(s+\tau)}_{1}+\norm{\psi_3(s+\tau)}_{q_1}^{q_1} 
 \right )
 ds ,
 $$
 which along with
 \eqref{pest1_4}-\eqref{pest1_9}
 concludes    the proof.
  \end{proof}

By   Lemma \ref{est1},  we  obtain  the following estimates.

\begin{lemma}
\label{est2} 
 Suppose  \eqref{f1}-\eqref{f3}  and \eqref{g1} hold.
Then for every $\alpha \le \alpha_0$, 
 $\tau \in \R$, $\omega \in \Omega$   and $D=\{D(\tau, \omega)
: \tau \in \R,  \omega \in \Omega\}  \in \cald$,
 there exists  $T=T(\tau, \omega,  D, \alpha)>0$ such that 
 for all $t \ge T$  and ,  the solution
 $v$ of  problem  \eqref{veq1}-\eqref{veq2}     satisfies 
 $$
\norm{v(\tau, \tau-t,\theta_{-\tau}\omega,v_{\tau-t})}^2 
$$
$$
+
\int_{\tau-t}^{\tau}e^{\frac{5}{4}\lambda(s-\tau)
-2\alpha\int_{0}^{s-\tau}
\eta(\theta_{r}\omega)dr}
\norm{v(s, \tau-t,\theta_{-\tau}\omega,v_{\tau-t})}^2ds 
$$
$$
+
\int_{\tau-t}^{\tau}e^{\frac{5}{4}
\lambda(s-\tau)-2\alpha\int_{0}
^{s-\tau}\eta(\theta_{r}\omega)dr}
\norm{\nabla (v(s, \tau-t,\theta_{-\tau}\omega,v_{\tau-t})
+\eps h(x)z(\theta_{s-\tau} \omega))}^p_pds 
$$
\be\label{est2_1}
+
\int_{\tau-t}^{\tau}e^{\frac{5}{4}
\lambda(s-\tau)-2\alpha\int_{0}
^{s-\tau}\eta(\theta_{r}\omega)dr}
\norm{v(s, \tau-t,\theta_{-\tau}\omega,
v_{\tau-t})+\eps h(x)z(\theta_{s-\tau} \omega))}^q_qds
$$
$$
\le R(\tau, \omega, \alpha, \eps),
\ee
where $v_{\tau -t}  \in D(\tau -t, \theta_{ -t} \omega)$
  and
  $R(\tau, \omega,  \alpha, \eps) $ is given by
$$
 R(\tau, \omega,  \alpha, \eps)
$$
$$
= c\int_{-\infty}^{0}e^{\frac{5}{4}
\lambda s -2\alpha\int_{0}^{s}
\eta(\theta_{r}\omega)dr} 
 \left ( \abs{\eps z(\theta_{s} \omega)}^p 
+ \abs{\eps z(\theta_{s} \omega)}^q 
+ \abs{\alpha \eps \eta(\theta_{s}\omega)z(\theta_{s}\omega )
}^2 \right )ds
$$
\be\label{est2_2}
+c\int_{-\infty}^{0}e^{\frac{5}{4}
\lambda s -2\alpha\int_{0}^{s}
\eta(\theta_{r}\omega)dr}  
(\norm{g(s+\tau)}^2 
+ \norm{\psi_1(s+\tau)}_{1}
+\norm{\psi_3(s+\tau)}_{q_1}^{q_1})ds,
\ee
 with 
  $c$  being  a  positive constant   independent of $\tau$, $\omega$,  $D$,
  $\alpha$ and $\eps$.
  In addition,   we have
   \be\label{est2_3}
  \lim_{t \to \infty}
  e^{-{\frac 54} \lambda t + 2 \alpha \int_{-t}^0 \eta (\theta_r\omega) dr}  
    R(\tau- t, \theta_{-t} \omega, \alpha, \eps) =0.
  \ee
\end{lemma}

\begin{proof}
\eqref{est2_1}  and \eqref{est2_2}
are special cases  of  Lemma \ref{est1}
for   $\sigma =\tau$.  We now  prove \eqref{est2_3}.
 By  \eqref{est2_2}  we  have
$$  R(\tau- t, \theta_{-t} \omega, \alpha, \eps)
$$
$$
 =
  c\int_{-\infty}^{0}e^{\frac{5}{4}
\lambda s -2\alpha\int_{0}^{s}
\eta(\theta_{r-t}\omega)dr} 
$$
$$
\times \left ( \abs{\eps z(\theta_{s-t} \omega)}^p 
+ \abs{\eps z(\theta_{s-t} \omega)}^q 
+ \abs{\alpha \eps \eta(\theta_{s-t}\omega)z(\theta_{s-t}\omega )
}^2 \right )ds
$$
$$
+c\int_{-\infty}^{0}e^{\frac{5}{4}
\lambda s -2\alpha\int_{0}^{s}
\eta(\theta_{r-t}\omega)dr}  
$$
$$
\times(\norm{g(s+\tau -t)}^2 
+ \norm{\psi_1(s+\tau-t)}_{1}
+\norm{\psi_3(s+\tau-t)}_{q_1}^{q_1})ds
$$
$$
 =
  c\int_{-\infty}^{-t}e^{\frac{5}{4}
\lambda (t+s) -2\alpha\int_{-t}^{s}
\eta(\theta_{r}\omega)dr} 
$$
$$
\times \left ( \abs{\eps z(\theta_{s} \omega)}^p 
+ \abs{\eps z(\theta_{s} \omega)}^q 
+ \abs{\alpha \eps \eta(\theta_{s}\omega)z(\theta_{s}\omega )
}^2 \right )ds
$$
$$
+c\int_{-\infty}^{-t}e^{\frac{5}{4}
\lambda (t+s) -2\alpha\int_{-t}^{s}
\eta(\theta_{r}\omega)dr}  
$$
$$
\times(\norm{g(s+\tau )}^2 
+ \norm{\psi_1(s+\tau)}_{1}
+\norm{\psi_3(s+\tau)}_{q_1}^{q_1})ds.
$$
Therefore we get
$$
  e^{-{\frac 54} \lambda t + 2 \alpha \int_{-t}^0 \eta (\theta_r\omega) dr}  
    R(\tau- t, \theta_{-t} \omega, \alpha, \eps) 
    $$
    $$
 =
  c\int_{-\infty}^{-t}e^{\frac{5}{4}
\lambda s  -2\alpha\int_{0}^{s}
\eta(\theta_{r}\omega)dr} 
 \left ( \abs{\eps z(\theta_{s} \omega)}^p 
+ \abs{\eps z(\theta_{s} \omega)}^q 
+ \abs{\alpha \eps \eta(\theta_{s}\omega)z(\theta_{s}\omega )
}^2 \right )ds
$$
\be\label{pest2_1}
+c\int_{-\infty}^{-t}e^{\frac{5}{4}
\lambda  s  -2\alpha\int_{0}^{s}
\eta(\theta_{r}\omega)dr}  
(\norm{g(s+\tau )}^2 
+ \norm{\psi_1(s+\tau)}_{1}
+\norm{\psi_3(s+\tau)}_{q_1}^{q_1})ds.
\ee\sloppy
Since the integrals in \eqref{est2_2} are convergent, 
by \eqref{pest2_1}
we obtain 
$ e^{-{\frac 54} \lambda t + 2 \alpha \int_{-t}^0 \eta (\theta_r\omega) dr}  
    R(\tau- t, \theta_{-t} \omega, \alpha, \eps) 
    \to 0$  as $t \to \infty$.
    This completes   the proof.
\end{proof}

Next,  we derive uniform estimates on the
tails of 
 solutions of \eqref{veq1}-\eqref{veq2}
 outside a bounded domain. These estimates  are
 crucial for proving the asymptotic compactness of solutions
 on unbounded domains.

 \begin{lemma}
 \label{est3}
Suppose \eqref{f1}-\eqref{f3} and \eqref{g1} hold. 
Then for every $\nu > 0$,  $\alpha \le \alpha_0$,
$\eps>0$, 
 $\tau \in \R$,
$\omega \in \Omega$  and  $D\in \cald$, 
there exists $T = T(\tau, \omega, D,  \alpha, \eps,  \nu) > 0$
 and $K = K(\tau, \omega,  \alpha, \eps,  \nu) \geq 1$
  such that for all $t \geq T$
  and $\sigma \in [\tau -1, \tau]$, 
  the solution $v$ of  
  \eqref{veq1}-\eqref{veq2}   satisfies
$$
\int_{\abs{x}\geq K}
\abs{v(\sigma,  \tau-t,\theta_{-\tau}\omega,v_{\tau-t})}^2dx
 \leq \nu,
$$
where $v_{\tau -t} \in D(\tau -t, \theta_{-t} \omega )$.
In addition,  $  T(\tau, \omega, D,  \alpha, \eps,  \nu) $
and $  K(\tau, \omega, D,  \alpha, \eps,  \nu) $
are  uniform  with respect to  $\eps \in (0,1]  $.
\end{lemma}

\begin{proof}
Let $\rho$ be a smooth
 function defined on $\R^{+}$ 
 such that $0\leq \rho(s) \leq 1$ for
  all $s \in \R^{+}$, and 
\begin{equation*} 
   \rho(s) = \left\{
     \begin{array}{ll}
       0  &  \text{ for } 0\leq s \leq 1;\\
       1  &  \text{ for } s \geq 2.
     \end{array}
   \right.
\end{equation*} 
 Multiplying \eqref{veq1}
  by $\rho(\frac{\abs{x}^2}{k^2})v$ and 
  then integrating over $\R^n$ we  get
 \be \label{pest3_1}
\frac{1}{2}\frac{d}{dt}\int_{\R^n}
 \rho(\frac{\abs{x}^2}{k^2})\abs{v}^2dx 
$$
$$
 - \int_{\R^n}
 \rho(\frac{\abs{x}^2}{k^2})\text{div}
 (\abs{\nabla (v+\eps h z(\theta_{t} \omega))}
 ^{p-2}\nabla (v+\eps h z(\theta_{t} \omega)))vdx
 $$
 $$
= ( \alpha\eta(\theta_t\omega)- \lambda )
 \ii \rho(\frac{\abs{x}^2}{k^2})\abs{v}^2dx
+ \int_{R^n}\rho(\frac{\abs{x}^2}{k^2})
 f(t,x,v+\eps h z(\theta_{t} \omega)) v dx
 $$ $$
 + \alpha \eps \eta (\theta_t\omega) z(\theta_t\omega) 
  \ii \rho(\frac{\abs{x}^2}{k^2}) hvdx 
+ \ii \rho(\frac{\abs{x}^2}{k^2})g(t,x)vdx.
\ee
For  the term involving the divergence  we have
 \be \label{pest3_3}
\int_{\R^n}\rho(\frac{\abs{x}^2}{k^2})
\text{div}(\abs{\nabla (v+\eps hz(\theta_{t} \omega))}
^{p-2}\nabla (v+\eps hz(\theta_{t} \omega)))vdx
$$
$$
= -\int_{\R^n}\rho(\frac{\abs{x}^2}{k^2})
\abs{\nabla (v+\eps hz(\theta_{t} \omega))}^{p}dx 
$$
$$
+ \int_{\R^n}\rho(\frac{\abs{x}^2}{k^2})
\abs{\nabla (v+\eps hz(\theta_{t} \omega))}
^{p-2} \nabla (v+\eps hz(\theta_{t} \omega))
\cdot\nabla (\eps hz(\theta_{t} \omega )) dx 
$$
$$
-\int_{\R^n}\rho'(\frac{\abs{x}^2}{k^2}) \frac{2x}{k^2}
\cdot\nabla (v+\eps hz(\theta_{t} \omega)) 
\abs{\nabla (v+\eps hz(\theta_{t} \omega))}^{p-2}vdx 
$$
$$
\leq 
- {\frac 12} \int_{\R^n}\rho(\frac{\abs{x}^2}{k^2})
\abs{\nabla (v+\eps hz(\theta_{t} \omega))}^{p}dx 
+c_1 \int_{\R^n}\rho(\frac{\abs{x}^2}{k^2})
\abs{\nabla(\eps hz(\theta_{t} \omega))}^{p}dx
$$
$$
-\int_{k \leq \abs{x} \leq 2k}
\rho'(\frac{\abs{x}^2}{k^2})
 \frac{2x}{k^2}\cdot\nabla 
(v+\eps hz(\theta_{t} \omega)) 
\abs{\nabla (v+\eps hz(\theta_{t}
 \omega))}^{p-2}vdx 
 $$
 $$
\leq c_1 \int_{\R^n}\rho(\frac{\abs{x}^2}
{k^2})\abs{\nabla(\eps hz(\theta_{t} \omega))}
^{p}dx + \frac{c_2}{k}(\norm{v}
^p_p+\norm{\nabla(v
+\eps hz(\theta_{t} \omega))}^p_p).
\ee
As in \eqref{plvk1_3},   for the nonlinearity $f$   we have
  \be 
  \label{pest3_5}
\int_{R^n}\rho(\frac{\abs{x}^2}{k^2})
f(t,x,v+\eps h z(\theta_{t} \omega))vdx 
 $$
 $$
\leq \int_{R^n}\rho(\frac{\abs{x}^2}{k^2})
f(t,x,v+\eps h z(\theta_{t} \omega))
(v+\eps h z(\theta_{t} \omega))dx 
$$
$$
+ \int_{\R^n}\rho(\frac{\abs{x}^2}{k^2})
\abs{f(t,x,v+\eps h z(\theta_{t} \omega))} \;
\abs{\eps h z(\theta_{t} \omega)}dx
$$
$$
\leq -\gamma\int_{\R^n}\rho(\frac{\abs{x}^2}
{k^2})\abs{v+\eps h z(\theta_{t} \omega))}^q dx 
+ \int_{\R^n}\rho(\frac{\abs{x}^2}{k^2})\psi_1(t,x)dx
$$
$$
 +
\int_{\R^n}\rho(\frac{\abs{x}^2}{k^2})
\psi_2(t,x)\abs{v+\eps h z(\theta_{t} \omega))}
^{q-1}\abs{\eps h z(\theta_{t} \omega)}dx 
$$
$$
+ \int_{\R^n}\rho(\frac{\abs{x}^2}{k^2})
\abs{\psi_3(t,x)\eps h z(\theta_{t} \omega)}dx
$$
$$
\leq -\frac{\gamma}{2}\int_{\R^n}
\rho(\frac{\abs{x}^2}{k^2})
\abs{v+\eps h z(\theta_{t} \omega))}^q dx
$$
$$
 + \int_{\R^n}\rho(\frac{\abs{x}^2}{k^2})
 (\abs{\psi_1(t,x)} + \abs{\psi_3(t,x)}^{q_1})dx 
 + c_3 \int_{\R^n}\rho(\frac{\abs{x}^2}{k^2})
 \abs{\eps h z(\theta_{t} \omega)}^q dx.
\ee
Note that
\be 
\label{pest3_7}
\alpha \eps \eta(\theta_t\omega)z(\theta_t\omega)
\int_{R^n}\rho(\frac{\abs{x}^2}{k^2})
hvdx + \int_{R^n}\rho(\frac{\abs{x}^2}{k^2})g(t,x)vdx
$$
$$
\le
  c_4 \int_{R^n}\rho(\frac{\abs{x}^2}{k^2})
  \abs{\alpha \eps \eta(\theta_t\omega)
  z(\theta_t\omega) h} ^2 dx
  $$
  $$
+ c_5
\int_{R^n}\rho(\frac{\abs{x}^2}{k^2})
\abs{g (t,x)}^2dx 
+ \frac{3}{8}\lambda\int_{R^n}
\rho(\frac{\abs{x}^2}{k^2})\abs{v}^2dx.
\ee
It follows   from \eqref{pest3_1}-\eqref{pest3_7} that 
 \be
 \label{pest3_11} 
\frac{d}{dt}\int \rho(\frac{\abs{x}^2}{k^2})\abs{v}^2dx
 + 
   (
 \frac{5}{4}\lambda -2\alpha \eta  (\theta_t\omega)
   )
 \int \rho(\frac{\abs{x}^2}{k^2})
 \abs{v}^2dx 
 $$
 $$
 \leq  
    \frac{c_7 }{k}(\norm{v}^p_p
  +\norm{\nabla(v+\eps hz(\theta_{t} \omega))}^p_p)
  $$
  $$
+ c_7 \int_{\R^n}\rho(\frac{\abs{x}^2}{k^2})
( \abs{g(t,x)}^2+ \abs{\psi_1(t,x)} + \abs{\psi_3(t,x)}^{q_1})dx
$$
$$
 + c_7 \int_{\R^n}\rho(\frac{\abs{x}^2}{k^2})
 (
  \abs{\nabla \eps h z(\theta_{t} \omega)}^p 
  +
 \abs{\eps h z(\theta_{t} \omega)}^q 
 +
 \abs{\alpha \eps \eta  (\theta_{t} \omega)  z(\theta_{t} \omega) h}^2 
 )dx. 
\ee
 Since $h \in H^1(\R^n)
 \bigcap 
  W^{1,q}(\R^n)$
  with $2\le p\le q$, 
  we  find  that for every $\nu > 0$, there exists a $
  K_1=K_1(\nu) \ge 1$ such that for all $k  \ge  K_1$,
 \be
 \label{pest3_13}
c_7 \int_{\R^n}\rho(\frac{\abs{x}^2}{k^2})
 (
  \abs{\nabla \eps h z(\theta_{t} \omega)}^p 
  +
 \abs{\eps h z(\theta_{t} \omega)}^q 
 +
 \abs{\alpha \eps \eta  (\theta_{t} \omega)  z(\theta_{t} \omega) h}^2
 )dx
 $$
 $$
 =c_7 
 \int_{\abs{x} \geq k}
 \rho(\frac{\abs{x}^2}{k^2})
 (
  \abs{\nabla \eps h z(\theta_{t} \omega)}^p 
  +
 \abs{\eps h z(\theta_{t} \omega)}^q 
 +
 \abs{\alpha \eps \eta  (\theta_{t} \omega)  z(\theta_{t} \omega) h}^2
 )dx
 $$
 $$
  \leq \nu
  \left (
    \abs{\eps z(\theta_t\omega)}^p
    + \abs{\eps z(\theta_t\omega)}^q
    + \abs{\alpha \eps \eta  (\theta_{t} \omega)  z(\theta_{t} \omega)  }^2
    \right ).
\ee
By \eqref{pest3_11}-\eqref{pest3_13} we find that
there exists $K_2 =K_2(\nu) \ge K_1$ such that
for all $k \ge K_2$, 
\be
 \label{pest3_15} 
\frac{d}{dt}\int \rho(\frac{\abs{x}^2}{k^2})\abs{v}^2dx
 + 
   (
 \frac{5}{4}\lambda -2\alpha \eta  (\theta_t\omega)
   )
 \int \rho(\frac{\abs{x}^2}{k^2})
 \abs{v}^2dx 
 $$
 $$
 \leq  
    \nu (\norm{v}^p_p
  +\norm{\nabla(v+\eps hz(\theta_{t} \omega))}^p_p)
  $$
  $$
+ c_7 \int_{\abs{x} \ge k} 
( \abs{g(t,x)}^2+ \abs{\psi_1(t,x)} + \abs{\psi_3(t,x)}^{q_1})dx
$$
$$
 + \nu
  \left (
    \abs{\eps z(\theta_t\omega)}^p
    + \abs{\eps z(\theta_t\omega)}^q
    + \abs{\alpha \eps \eta  (\theta_{t} \omega)  z(\theta_{t} \omega)  }^2
    \right ).
\ee
Multiplying \eqref{pest3_15} 
by $e^{\frac{5}{4}\lambda t-2\alpha\int^{t}
_{0}\eta(\theta_{r}\omega)dr}$, 
and integrating from $\tau -t$ to $\sigma$
with $\sigma \ge \tau -t$,
 we get 
 \be
 \label{pest3_17}
\int_{\R^n} \rho(\frac{\abs{x}^2}{k^2})
\abs{v(\sigma, \tau-t,\omega,v_{\tau-t})}^2dx
$$
$$
 \leq e^{\frac{5}{4} \lambda (\tau -t -\sigma) 
 -2\alpha\int_{\sigma}^{\tau-t}
 \eta(\theta_{r}\omega)dr}
 \int_{\R^n} \rho(\frac{\abs{x}^2}{k^2})
 \abs{v_{\tau-t}}^2dx
 $$
 $$
+ \nu \int_{\tau-t}^{\sigma}
e^{\frac{5}{4}\lambda(s- \sigma)
-2\alpha\int_{ \sigma }^{s}
\eta(\theta_{r}\omega)dr}
(\norm{v(s, \tau-t,\omega,v_{\tau-t})}
^p_p+\norm{\nabla(v 
+\eps hz(\theta_{s} \omega))}^p_p)ds
$$
$$
+ \nu \int_{\tau-t}^{ \sigma }
e^{\frac{5}{4}\lambda(s- \sigma  )
-2\alpha\int_{\sigma }^{s}\eta(\theta_{r}\omega)dr}
(
\abs{\eps z(\theta_s\omega)}^p +
\abs{\eps z(\theta_s\omega)}^q 
+ \abs{\alpha \eps \eta(\theta_s\omega)
z(\theta_s\omega)}^2 
)ds
$$
$$
+ c_7 \int_{\tau-t}^{ \sigma  }\int_{\abs{x} \ge k }
e^{\frac{5}{4}\lambda(s- \sigma  )
-2\alpha\int_{\sigma   }^{s}\eta(\theta_{r}\omega)
dr}(\abs{g(s,x)}^2 +\abs{\psi_1(s,x)} 
+ \abs{\psi_3(s,x)}^{q_1}
 )dxds.
\ee
 Replacing $\omega$ with
  $\theta_{-\tau}\omega$ in \eqref{pest3_17},
  after simple calculations,  we  get
  for all $k\ge K_2$  and $\sigma \in [\tau -1, \tau]$,
 \be
 \label{pest3_21}
\int_{\R^n} \rho(\frac{\abs{x}^2}{k^2})
\abs{v( \sigma , \tau-t,\theta_{-\tau}\omega,v_{\tau-t})}^2dx
 \leq e^{\frac{5}{4}\lambda (\tau -t -\sigma)
 +2\alpha\int^{\sigma -\tau}_{-t}
 \eta(\theta_{r}\omega)dr}  
 \norm{v_{\tau-t}}^2 
 $$
 $$
 + \nu \int_{\tau-t}^{ \sigma }
e^{\frac{5}{4}\lambda(s- \sigma )
-2\alpha\int_{\sigma -\tau}^{s-\tau}
\eta(\theta_{r}\omega)dr}
$$
$$
\times(\norm{v(s, \tau-t,  \theta_{-\tau} \omega,v_{\tau-t})}
^p_p+\norm{\nabla(v 
+\eps hz(\theta_{s-\tau} \omega))}^p_p)ds
$$
  $$
+ \nu \int_{-\infty}^{\sigma -\tau}e^{\frac{5}{4}\lambda 
(s +\tau -\sigma) 
 - 2\alpha\int^{s}_{\sigma -\tau}\eta(\theta_{r}\omega)dr}
$$
$$
\times (
 \abs{\eps z(\theta_s\omega)}^p
 +
 \abs{\eps z(\theta_s\omega)}^q 
 + \abs{\alpha \eps \eta(\theta_s\omega)
 z(\theta_s\omega)}^2 
 )ds
 $$
 $$
+ c_7 \int_{-\infty}^{\sigma -\tau}\int_{\abs{x} \ge k}
e^{\frac{5}{4}\lambda  (s +\tau -\sigma) 
-2\alpha\int^{s}_{\sigma -\tau}\eta(\theta_{r}\omega)dr}
$$
$$
\times(\abs{g(s+\tau,x)}^2 
+\abs{\psi_1(s+\tau,x)}
+ \abs{\psi_3(s+\tau,x)}^{q_1}
 )dxds
 $$
 $$
  \leq 
 e^{\frac{5}{4}\lambda  
+2\alpha\int_{-1}^{0}
\abs{\eta(\theta_{r}\omega)}dr}
 e^{-\frac{5}{4}\lambda t
+2\alpha\int^{0}_{-t}\eta(\theta_{r}\omega)dr}  
 \norm{v_{\tau-t}}^2 
 $$
 $$
 + \nu e^{\frac{5}{4}\lambda  
+2\alpha\int_{-1}^{0}
\abs{\eta(\theta_{r}\omega)}dr}
  \int_{\tau-t}^{ \tau }
e^{\frac{5}{4}\lambda(s- \tau )
-2\alpha\int_{0}^{s-\tau}
\eta(\theta_{r}\omega)dr}
$$
$$
\times(\norm{v }
^p_p+\norm{\nabla(v 
+\eps hz(\theta_{s-\tau} \omega))}^p_p)ds
$$
  $$
+ \nu
 e^{\frac{5}{4}\lambda  
+2\alpha\int_{-1}^{0}
\abs{\eta(\theta_{r}\omega)}dr}
  \int_{-\infty}^{0}e^{\frac{5}{4}\lambda s
 - 2\alpha\int^{s}_{0}\eta(\theta_{r}\omega)dr}
$$
$$
\times(
 \abs{\eps z(\theta_s\omega)}^p
 +
 \abs{\eps z(\theta_s\omega)}^q 
 + \abs{\alpha \eps \eta(\theta_s\omega)
 z(\theta_s\omega)}^2 
 )ds
 $$
 $$
+ c_8 \int_{-\infty}^{0}\int_{\abs{x} \ge k}
e^{\frac{5}{4}\lambda s
-2\alpha\int^{s}_{0}\eta(\theta_{r}\omega)dr}
$$
$$
\times(\abs{g(s+\tau,x)}^2 
+\abs{\psi_1(s+\tau,x)}
+ \abs{\psi_3(s+\tau,x)}^{q_1}
 )dxds.
\ee
 Since $v_{\tau-t} \in D(\tau-t,\theta_{-\tau}\omega)$, 
 we see that for every $\nu > 0$, 
 $\tau \in \R$,
 $\omega \in \Omega$
 and $\alpha>0$,
 there exists a $T_1(\tau,\omega,D, \alpha, \nu) > 0$ 
 such that for every $t \ge T_1$ and $\sigma \in [\tau-1, \tau]$,
 \be\label{pest3_23}
  e^{\frac{5}{4}\lambda  
+2\alpha\int_{-1}^{0}
\abs{\eta(\theta_{r}\omega)}dr}
 e^{-\frac{5}{4}\lambda t
+2\alpha\int^{0}_{-t}\eta(\theta_{r}\omega)dr}  
 \norm{v_{\tau-t}}^2 
 $$
 $$ \le
 e^{\frac{5}{4}\lambda  
+2\alpha\int_{-1}^{0}
\abs{\eta(\theta_{r}\omega)}dr}
 e^{-\frac{5}{4}\lambda t
+2\alpha\int^{0}_{-t}\eta(\theta_{r}\omega)dr}  
 \norm{D({\tau-t}, \theta_{-t} \omega)}^2 
  \leq \nu .
\ee\sloppy
Since  $\int_{-\infty}^{0}\ii
e^{\frac{5}{4}\lambda s
-2\alpha\int^{s}_{0}\eta(\theta_{r}\omega)dr}
(\abs{g(s+\tau,x)}^2 
+\abs{\psi_1(s+\tau,x)}
+ \abs{\psi_3(s+\tau,x)}^{q_1}
 )dxds$
 is convergent, we have
$$
 \int_{-\infty}^{0}\int_{\abs{x} \ge k}
e^{\frac{5}{4}\lambda s
-2\alpha\int^{s}_{0}\eta(\theta_{r}\omega)dr}
$$
$$
\times(\abs{g(s+\tau,x)}^2 
+\abs{\psi_1(s+\tau,x)}
+ \abs{\psi_3(s+\tau,x)}^{q_1}
 )dxds \to 0,
$$
 as $k \to \infty$. 
 Therefore,  there exists
 $K=K_3(\tau, \omega, \alpha, \nu) \ge K_2$
 such that  for all $k \ge K_3$,
 \be\label{pest3_25}
c_8  \int_{-\infty}^{0}\int_{\abs{x} \ge k}
e^{\frac{5}{4}\lambda s
-2\alpha\int^{s}_{0}\eta(\theta_{r}\omega)dr}
$$
$$
\times(\abs{g(s+\tau,x)}^2 
+\abs{\psi_1(s+\tau,x)}
+ \abs{\psi_3(s+\tau,x)}^{q_1}
 )dxds  \le \nu.
 \ee
 Note that
 $$
  \norm{v(s, \tau-t,  \theta_{-\tau} \omega,v_{\tau-t})}
^p_p
$$
$$
\le
2^{p}
\left (\norm{v(s, \tau-t,  \theta_{-\tau} \omega,v_{\tau-t})
+ \eps h z(\theta_{s-\tau}  \omega) }
^p_p
+ 
 \norm{ 
  \eps h z(\theta_{s-\tau} \omega) }
^p_p
\right ),
$$\sloppy
which along with
\eqref{p_inequality} and
 Lemma \ref{est2} shows that
 there exists $T_2
 = T_2(\tau, \omega, D, \alpha,  \nu) \ge T_1$
 such that  for all $t \ge T_2$,
 $$
   \int_{\tau-t}^{\tau}
e^{\frac{5}{4}\lambda(s-\tau)
-2\alpha\int_{0}^{s-\tau}
\eta(\theta_{r}\omega)dr}
(\norm{v(s, \tau-t,  \theta_{-\tau} \omega,v_{\tau-t})}
^p_p+\norm{\nabla(v 
+\eps hz(\theta_{s-\tau} \omega))}^p_p)ds
$$
\be\label{pest3_30}
\le  
c_9 R(\tau, \omega, \alpha, \eps)
+    c_{10} 
 \int_{-\infty}^{0}e^{\frac{5}{4}\lambda s
 - 2\alpha\int^{s}_{0}\eta(\theta_{r}\omega)dr}
 \abs{\eps z(\theta_s\omega)}^p ds, 
\ee
where 
$R(\tau, \omega, \alpha, \eps)$
is  the number given by \eqref{est2_2}.
It follows  from \eqref{pest3_21}-\eqref{pest3_30}
that for all $k\ge K_3$,   $t \ge T_2$
and $\sigma \in [\tau -1, \tau]$,
\be
 \label{pest3_32}
\int_{\R^n} \rho(\frac{\abs{x}^2}{k^2})
\abs{v(\sigma, \tau-t,\theta_{-\tau}\omega,v_{\tau-t})}^2dx
 \le 2\nu
 +  \nu
c_{11} R(\tau, \omega, \alpha, \eps)
$$
$$
+ \nu c_{12}
    \int_{-\infty}^{0}e^{\frac{5}{4}\lambda s
 - 2\alpha\int^{s}_{0}\eta(\theta_{r}\omega)dr}
 (
 \abs{\eps z(\theta_s\omega)}^p
 +
 \abs{\eps z(\theta_s\omega)}^q 
 + \abs{\alpha \eps \eta(\theta_s\omega)
 z(\theta_s\omega)}^2 
 )ds.
\ee
Note that  $\rho(\frac{\abs{x}^2}{k^2})= 1$
when $\abs{x}^2 \ge 2 k^2$. This along with
\eqref{pest3_32}  concludes    the proof.
   \end{proof}

The asymptotic compactness of solutions of equation  \eqref{veq1}
is  given  below.

\begin{lemma}
\label{asyv}
Suppose  \eqref{f1}-\eqref{f3}  and \eqref{g1} hold.
Then for every $\alpha \le \alpha_0$,
$\eps >0$,
 $\tau \in \R$, $\omega \in \Omega$   and $D=\{D(\tau, \omega)
: \tau \in \R,  \omega \in \Omega\}  \in \cald$,
 the sequence
 $v(\tau, \tau -t_n,  \theta_{-\tau} \omega,   v_{0,n}  ) $   has a
   convergent
subsequence in $\ltwo $ provided
  $t_n \to \infty$ and 
 $ v_{0,n}  \in D(\tau -t_n, \theta_{ -t_n} \omega )$.
 \end{lemma}
 
 \begin{proof}
 By Lemma \ref{est1} we find  that 
 there  exists  $N_1 =N_1(\tau, \omega, D, \alpha)>0$
  such that   for all $n \ge N_1$,
   \be\label{pav_1}
   \| v(\tau -1,  \tau -t_n, \theta_{-\tau} \omega, v_{0,n} ) \|
   \le c_1.
  \ee\sloppy
  Applying Lemma \ref{comv1} to the sequence
  $v(\tau,  \tau - 1, \theta_{-\tau} \omega,
    v(\tau -1,  \tau -t_n, \theta_{-\tau} \omega, v_{0,n} )  )$,
    we find  that there exist $s_0 \in (\tau -1, \tau)$, 
    $v_0 \in \ltwo$
    and a subsequence (not relabeled) such that
    as $n \to \infty$,
   $$
    v(s_0,  \tau - 1, \theta_{-\tau} \omega,
    v(\tau -1,  \tau -t_n, \theta_{-\tau} \omega, v_{0,n} )  )
    \to v_0
    \quad \text{ in } L^2(\o_k)
    \text{ for every }  k \in \N,
   $$
   that is,   as $n \to \infty$,
    \be
    \label{pav_2}
    v(s_0,    \tau -t_n, \theta_{-\tau} \omega, v_{0,n} )
    \to v_0
    \quad \text{ in } L^2(\o_k)
    \text{ for every }  k \in \N.
    \ee
    By \eqref{pexiv_30}
   we get
    $$
    \norm{ v(\tau,  s_0, \theta_{-\tau} \omega,
    v(s_0,  \tau -t_n, \theta_{-\tau} \omega, v_{0,n} ) ) 
    - v(\tau, s_0,  \theta_{-\tau} \omega, v_0)} 
    $$
    $$
    \le
    e^{c_1 (\tau -s_0)}
    \norm{  
   v(s_0,  \tau -t_n, \theta_{-\tau} \omega, v_{0,n} )  -v_0  
    }.
    $$  Since $s_0 \in (\tau-1, \tau)$, we obtain
     \be\label{pav_2a}
    \norm{ v(\tau,  s_0, \theta_{-\tau} \omega,
    v(s_0,  \tau -t_n, \theta_{-\tau} \omega, v_{0,n} ) ) 
    - v(\tau, s_0,  \theta_{-\tau} \omega, v_0)} ^2
    $$
    $$
    \le
     e^{2c_1  }
    \int_{\abs{x} < k}
     \abs{v(s_0,  \tau -t_n, \theta_{-\tau} \omega, v_{0,n} )
     -v_0 }^2 dx
     $$
     $$
     +
      e^{2c_1  }
    \int_{\abs{x} \ge k}
     \abs{v(s_0,  \tau -t_n, \theta_{-\tau} \omega, v_{0,n} )
     -v_0 }^2 dx
     $$
      $$
    \le
     e^{2c_1  }
    \int_{\abs{x} < k}
     \abs{v(s_0,  \tau -t_n, \theta_{-\tau} \omega, v_{0,n} )
     -v_0 }^2 dx
     $$
     $$
     +
     2 e^{2c_1  }
    \int_{\abs{x} \ge k}
    \left (
     \abs{v(s_0,  \tau -t_n, \theta_{-\tau} \omega, v_{0,n} )}^2
     + \abs{v_0}^2 
     \right )
       dx.
     \ee
     Since  $v_0 \in \ltwo$, given $\nu>0$,  there exists
     $K_1 =K_1(\nu) \ge 1$ such that for all $k \ge K_1$,
     \be\label{pav_3}
      2 e^{2c_1  }
    \int_{\abs{x} \ge k}
      \abs{v_0}^2  ds \le \nu.
      \ee
   On the other hand,  by Lemma \ref{est3},
   there exist $N_2=N_2(\tau, \omega, D,  \alpha, \eps, \nu)\ge 1$ 
     and
     $K_2=K_2(\tau, \omega,   \alpha, \eps, \nu)\ge  K_1$ 
     such that  for all $n \ge N_2$  and $k \ge K_2$,
     \be\label{pav_4}
       2  e^{2c_1  }
    \int_{\abs{x} \ge  k}
     \abs{v(s_0,  \tau -t_n, \theta_{-\tau} \omega, v_{0,n} )
       }^2 dx
     \le \nu.
     \ee
     By \eqref{pav_2} we find that there exists
     $N_3 = N_3(\tau, \omega, D,  \alpha, \eps, \nu)\ge N_2$ 
     such that
     for all $n \ge N_3$,
     \be\label{pav_8}
      e^{2c_1  }
    \int_{\abs{x} < K_2}
     \abs{v(s_0,  \tau -t_n, \theta_{-\tau} \omega, v_{0,n} )
     -v_0 }^2 dx
     \le \nu.
     \ee
     It follows   from \eqref{pav_2a}-\eqref{pav_8}
     that  for all $n \ge N_3$, 
       $$
    \norm{ v(\tau,  s_0, \theta_{-\tau} \omega,
    v(s_0,  \tau -t_n, \theta_{-\tau} \omega, v_{0,n} ) ) 
    - v(\tau, s_0,  \theta_{-\tau} \omega, v_0)} ^2
    \le 3\nu,
    $$
    that is, for all $n \ge N_3$, 
      $$
    \norm{ v(\tau,   \tau -t_n, \theta_{-\tau} \omega, v_{0,n}  ) 
    - v(\tau, s_0,  \theta_{-\tau} \omega, v_0)} ^2
    \le 3\nu.
    $$
    Therefore, 
    $ v(\tau,   \tau -t_n, \theta_{-\tau} \omega, v_{0,n}  ) $
    converges to  $  v(\tau, s_0,  \theta_{-\tau} \omega, v_0) $
    in $\ltwo$. This completes   the proof.
   \end{proof}

\section{Random Attractors}
\setcounter{equation}{0}

  In this section, we prove the existence of $\cald$-pullback attractor
  for \eqref{seq1}-\eqref{seq2} in $\ltwo$
  by  Proposition \ref{att}.
  To this end, we need to establish   the
  existence of    $\cald$-pullback absorbing sets
  and the $\cald$-pullback asymptotic compactness
  of $\Phi$ in $\ltwo$. 
  The existence of  absorbing sets of $\Phi$ is given below.

  \begin{lemma}
  \label{lem41}
    Suppose \eqref{f1}-\eqref{f3}
    and \eqref{g1} hold. Then    for every
    $\alpha \le \alpha_0$  and $\eps >0$,  
    the stochastic equation \eqref{seq1}
    with \eqref{seq2}  has   a closed
    measurable $\cald$-pullback absorbing set $K=\{ K(\tau, \omega):
    \tau \in \R, \omega \in \Omega \} \in \cald$ which  is given by
    \be\label{lem41_1}
    K (\tau, \omega) = \{ u \in \ltwo: \| u\|^2 \le
    2\norm{\eps h z(\omega)}^2
    +
    2 R(\tau, \omega, \alpha, \eps)
    \},
    \ee
    where $R(\tau, \omega, \alpha, \eps)$ is the number given by
    \eqref{est2_2}.
  \end{lemma}
  
  \begin{proof}
  Let $D =\{ D(\tau, \omega): \tau \in \R, \omega \in
  \Omega\} \in \cald$. For every $\tau \in \R$   and
  $\omega \in \Omega$, denote by
 \be\label{plem41_1a}
  {\widetilde{D}}(\tau, \omega)
  =\{ v \in \ltwo: v= u -\eps h z(\omega)
  \text{ for some } u \in D(\tau, \omega) \}.
  \ee
  Since $z$ is tempered, we find that
  the family 
  ${\widetilde{D}}
  = \{{\widetilde{D}}(\tau, \omega), \tau \in \R, \omega \in
  \Omega  \}$ belongs to $\cald$ provided
  $D\in \cald$.
  By \eqref{uv} we have
  \be\label{plem41_1}
  u(\tau, \tau -t, \theta_{-\tau} \omega, u_{\tau -t})
$$
$$
  = 
   v(\tau, \tau -t, \theta_{-\tau} \omega, v_{\tau -t})
   + \eps h z(\omega)  \ 
   \text{ with } \  v_{\tau -t}
   = u_{\tau -t} -\eps h z(\theta_{-t} \omega).
   \ee
   Thus, if $u_{\tau -t}
   \in D(\tau-t, \theta_{-t} \omega) \in \cald$, then
   $v_{\tau -t}
   \in {\widetilde{D}}(\tau-t, \theta_{-t} \omega)
   \in \cald$. 
   By Lemma \ref{est2} we find that
   there exists  $T=T(\tau, \omega, D, \alpha, \eps) >0$
   such that for all $t \ge T$,
    $$\norm{v(\tau, \tau -t, \theta_{-\tau} \omega, v_{\tau -t})}
    ^2
    \le  R(\tau, \omega, \alpha, \eps),
    $$
    where $R(\tau, \omega, \alpha, \eps)$
    is as in \eqref{est2_2}. 
    By \eqref{plem41_1} we get for all $t \ge T$,
    $$
    \norm{u(\tau, \tau -t, \theta_{-\tau} \omega, u_{\tau -t})}
    ^2
    \le  2 \norm{\eps h z(\omega) }^2
    + 2R(\tau, \omega, \alpha, \eps ).
    $$
    This along with 
    \eqref{pcycle2} and \eqref{lem41_1} shows
     that  for all $t \ge T$,
    \be\label{plem41_4}
    \Phi (t, \tau -t, \theta_{-t} \omega,
    D(\tau-t, \theta_{-t} \omega) )
    \subseteq K(\tau, \omega).
    \ee
    On the other hand, by \eqref{est2_3} and the temperedness
    of $z$ we obtain
    \be\label{plem41_6}
    \lim_{t \to \infty}
  e^{-{\frac 54} \lambda t + 2 \alpha \int_{-t}^0 \eta (\theta_r\omega) dr}  
  \norm{
    K(\tau- t, \theta_{-t} \omega)} =0.
    \ee
    By \eqref{plem41_4}-\eqref{plem41_6} we find that
    $K $ given by \eqref{lem41_1} is a closed
    $\cald$-pullback absorbing set of $\Phi$ in $\cald$.
    Note that   the measurability of  $K (\tau, \omega)$ 
    in $\omega \in \Omega$  follows  from that of 
      $z(\omega)$  and 
  $R(\tau, \omega, \alpha, \eps)$  immediately.
  This  completes     the proof.
     \end{proof}

    The following is our main result regarding the 
     existence of
    $\cald$-pullback attractors of $\Phi$.
    
    \begin{theorem}
    \label{eatt}
    Suppose  \eqref{f1}-\eqref{f3}  and \eqref{g1} hold.
    Then for every $\alpha \le \alpha_0$  and $\eps>0$,
 the   stochastic equation \eqref{seq1}
 with  \eqref{seq2}  
   has a unique $\cald$-pullback attractor $\cala
   =\{\cala(\tau, \omega):
      \tau \in \R, \ \omega \in \Omega \} \in \cald$
 in $\ltwo$.   
 In addition,  if there is $T>0$  such that
 $f(t,x,s)$, $g(t,x)$,
 $\psi_1(t,x)$ and $\psi_3 (t,x)$ are all
 $T$-periodic in $t$
 for fixed $x \in \R^n$
 and $s\in \R$,   then the attractor $\cala$
 is also $T$-periodic.
 \end{theorem}

\begin{proof}
We first prove  that  $\Phi$ is    $\cald$-pullback
asymptotically  compact    in $\ltwo$;
that is, 
  for every  
 $\tau \in \R$, $\omega \in \Omega$,
 $D  \in \cald$,  $t_n \to \infty$
 and $ u_{0,n}  \in D(\tau -t_n, \theta_{ -t_n} \omega )$,
we want  to show that 
 the sequence
 $\Phi (t_n, \tau -t_n,  \theta_{-t_n} \omega,   u_{0,n}  )$
  has a
   convergent
subsequence in $\ltwo $.
 Let  $ v_{0,n} 
 = u_{0,n} - \eps hz(\theta_{-t_n} \omega)$
 and  $\widetilde{D}$ be the family
  given by \eqref{plem41_1a}.
 Since   $ u_{0,n}  \in D(\tau -t_n, \theta_{ -t_n} \omega )$, we find
 that $ v_{0,n} 
   \in {\widetilde{D}}(\tau -t_n, \theta_{ -t_n} \omega )
  \in \cald$. Therefore,  by 
     \eqref{plem41_1}  and 
    Lemma \ref{asyv} we find that
  $u (\tau, \tau -t_n,  \theta_{-\tau} \omega,   u_{0,n}  )$
  has a convergent subsequence in $\ltwo$.
   This together with \eqref{pcycle2} indicates that
    $\Phi (t_n, \tau -t_n,  \theta_{-t_n} \omega,   u_{0,n}  )$
  has a
   convergent
subsequence,  and thus it is
$\cald$-pullback asymptotically compact  in $\ltwo $.
Since $\Phi$ also has a
closed measurable  $\cald$-pullback
absorbing  set $K$ given by \eqref{lem41_1},
by Proposition \ref{att} we get the existence and uniqueness
of $\cald$-pullback attractor $\cala \in \cald$ of $\Phi$
immediately.

Next, we discuss   $T$-periodicity of $\cala$.
Note  that  if  $f$  and $g$ are $T$-periodic in
their first arguments, then the cocycle $\Phi$
is also  $T$-periodic. 
Indeed,  in this case,    
for every
 $t \in \R^+$, $\tau \in \R$ and $\omega \in \Omega$, 
 by \eqref{pcycle2} we have
\be\label{peatt_1}
\Phi (t, \tau +T, \omega,  \cdot )
= u(t+ \tau +T, \tau +T,  \theta_{ -\tau -T} \omega,  \cdot )
$$
$$
=u(t +\tau, \tau, \theta_{ -\tau} \omega,  \cdot )
= \Phi (t, \tau,  \omega,  \cdot ).
\ee
 In addition,  if  $g(t,x)$, $\psi_1(t,x)$ and $\psi_3(t,x)$ are
 all $T$-periodic in $t$, then by
 \eqref{est2_2} and \eqref{lem41_1} we  get
  $ 
K(\tau +T, \omega)
=K(\tau, \omega )$
 for all $\tau \in \R$ and
$\omega \in \Omega$.
This along with
\eqref{peatt_1}
and Proposition \ref{att}
yields  the
$T$-periodicity of $\cala$.
 \end{proof}

\chapter{Multiplicative Noise}\label{mult_chap}
Given $\tau \in\R$ and $\omega \in \Omega$, 
consider the  following  stochastic     equation 
defined for   $ x \in {\R^n}$ and   $t > \tau$, 
\be
  \label{seq11}
  {\frac {\partial u}{\partial t}}  + \lambda u
  - {\rm div} \left (|\nabla u |^{p-2} \nabla u \right )
  =f(t,x,u ) 
  +g(t,x) 
  +\alpha u\circ {\frac {dW}{dt}}
 \ee  with  initial condition
 \be\label{seq22}
 u( \tau, x ) = u_\tau (x),   \quad x\in  {\R^n},
 \ee
 where $p\ge 2$,  $ \alpha>0  $,  $\lambda>0$,
$g \in L^2_{loc}(\R, \ltwo)$, and
$W$  is a  two-sided real-valued Wiener process on  $(\Omega, \calf, P)$.
We  assume the nonlinearity 
$f: \R \times {\R^n} \times \R$ 
$\to \R$ is continuous and satisfies, 
for all
$t, s \in \R$   and  $x \in {\R^n}$, 
\be 
\label{f11}
f (t, x, s) s \le - \gamma
 |s|^q + \psi_1(t, x),  
\ee
\be 
\label{f21}
|f(t, x, s) |   \le \psi_2 (t,x)  |s|^{q-1} + \psi_3 (t, x),
\ee
\be 
\label{f31}
{\frac {\partial f}{\partial s}} (t, x, s)   \le \psi_4 (t,x),
\ee\sloppy
where $\gamma>0$  and $ q \ge p $
are constants,
$\psi_1 \in L^1_{loc} (\R,  L^1( {\R^n}) )$, 
$\psi_2, \psi_4 \in  L^\infty_{loc} (\R, L^\infty ( {\R^n}))$,
and $\psi_3 \in L^{q_1}_{loc} (\R, L^{q_1} (  {\R^n}))$.
From now on,   we  always  
use $p_1$ and $q_1$
to denote 
  the conjugate exponents
of $p$ and $q$, respectively.

To define a random dynamical system for
  \eqref{seq11},  we need to transfer the stochastic equation
  to a pathwise  deterministic  system.
  As usual,  let $z$ be the 
  random variable  given by:
$$
z ( \omega)=   -\int_{-\infty}^0e^{\tau}\omega(\tau)d\tau, \quad \omega \in \Omega
$$
Then $z$ solves the following stochastic equation:

\be
\frac{d}{dt}z(\theta_t\omega)+z(\theta_t\omega)=\frac{dW}{dt}
\ee
It  follows from \cite{arn1} that
 there exists a $\theta_t$-invariant 
 set $\widetilde{\Omega}  $  
 of full
 measure 
 such that  
  $z(\theta_t \omega)$  is
 continuous in $t$ 
and 
\be\label{zergodic}
 \lim\limits_{t \to \pm \infty}   {\frac { z(\theta_t \omega)}{t}}
= 0 
\ \ 
\mbox{and} \ \ 
\lim_{t \to \pm \infty}
{\frac 1t} \int_0^t z(\theta_r \omega ) dr =0
\ee
for all  $\omega \in \widetilde{\Omega}$.
For convenience,  we will  denote  
$\widetilde{\Omega}$    by   $\Omega$
in the sequel. Let $u(t, \tau, \omega, u_\tau)$
   be a solution of problem \eqref{seq11}-\eqref{seq22}
   with initial condition $u_\tau$ at initial time $\tau$,  and  define
 \be 
 \label{uv1}
 v(t, \tau, \omega, v_\tau)
   =    e^{-\alpha z(\theta_{t} \omega)}u(t, \tau, \omega, u_\tau)
 \quad \mbox{with }   \
 v_\tau =   e^{-\alpha z(\theta_{\tau} \omega)}u_\tau.
 \ee
 By \eqref{seq11} and \eqref{uv1}, 
 after simple calculations,
  we get
 \be 
 \label{veq11}
\frac{\partial v}{\partial t}-e^{\alpha(p-2)z(\theta_{t} \omega)}\text{div}
 \left (
\abs{\nabla
 v}^{p-2}
 \nabla v
 \right )
 + \lambda v 
$$
$$
 = \alpha z(\theta_{t} \omega)v + e^{-\alpha z(\theta_{t} \omega)}f(t,x,e^{\alpha z(\theta_{t} \omega)}v)
+ e^{-\alpha z(\theta_{t} \omega)}g(t,x),
\ee
with initial   condition
\be 
 \label{veq22}
v(\tau, x )=v_{\tau}(x ), \quad  x\in {\R^n}.
\ee
The well-posedness of the equations is investigated in \cite{krause2}, from which we also have the following result which is necessary to prove asymptotic compactness of the cocycle.

\begin{lemma}
 \label{comv11}
 Let   \eqref{f11}-\eqref{f31}  hold and
 $\{v_n\}_{n=1}^\infty$ be a bounded sequence
 in $\ltwo$.
 Then for every   
  $\tau  \in \R$,  $t >\tau$ and 
 $\omega \in \Omega$,
 there exist  $v_0 \in L^2(\tau, t;   \ltwo)$
 and  a subsequence 
 $\{v(\cdot, \tau, \omega, v_{n_m})\}_{m=1}^\infty$
  of
 $\{v(\cdot, \tau, \omega, v_n)\}_{n=1}^\infty$
 such that 
 $v(s, \tau, \omega, v_{n_m})
 \to v_0 (s) $ in $L^2(\o_k)$
 as $m \to \infty$
 for every  fixed $k \in \N$
 and for almost all $s \in (\tau, t)$.
 \end{lemma}

We must also specify the attraction domain for the multiplicative case. Here, we  consider
    a family  
  $D =\{ D(\tau, \omega) \subseteq \ltwo: \tau \in \R, \omega \in \Omega \}$ 
  of   
  bounded nonempty    sets  such that
  for every $\tau \in \R$   and $\omega \in \Omega$, 
 \be
 \label{Dom11}
 \lim_{s \to  - \infty} e^{ \frac{5}{4}\lambda s -2\alpha\int_{0}^{s}z(\theta_{\xi}\omega)d\xi -2\alpha z(\theta_{s}\omega)}
  \| D( \tau + s, \theta_{s} \omega ) \|^2 =0,
\ee   where 
   $\| S \|=  \sup\limits_{u \in S}
   \| u\|_{\ltwo }$ for a nonempty bounded  subset $S$ of $\ltwo$.
  In  the sequel,   we will   use  $\cald$
   to denote the collection of all families 
   with property \eqref{Dom11}:
 \be
 \label{Dom21}
\cald  = \{ 
   D =\{ D(\tau, \omega) \subseteq \ltwo:
    \tau \in \R, \omega \in \Omega \}: \ 
 D  \ \mbox{satisfies} \  \eqref{Dom11} \} .
\ee
For our purpose,   we 
further assume  the following condition
on $g$, $\psi_1$  and $\psi_3$:
for every $ \tau \in \R$, 
 \be\label{g11}
\int_{-\infty}^\tau e^{\lambda s}
\left (
\| g(s, \cdot) \|^2 + \| \psi_1 (s, \cdot) \|_{L^1({\R^n})}
+ \| \psi_3 (s, \cdot) \|_{L^{q_1}({\R^n})}^{q_1}
\right ) ds < \infty.
\ee 

\section{ Existence of Random Attractors}
\setcounter{equation}{0}

In this section, we prove the existence and uniqueness of
$\cald$-pullback attractors for 
the stochastic  equation   \eqref{seq11}.
We also show the periodicity of the attractor when
external terms are periodic in time.
We start with the uniform  estimates of solutions in $\ltwo$.

\begin{lemma}
\label{est11}
If  \eqref{f11}-\eqref{f31} and \eqref{g11} hold, 
then   for every $\alpha > 0$,
$\tau \in \mathbb{R}$,
 $\sigma \in [\tau-1,\tau]$,
   $\omega \in \Omega$, 
  and $D = \{D(\tau,\omega):\tau\in\mathbb{R},\omega\in\Omega\}
   \in \cald$, there is  $T = T(\tau, \omega, D, \alpha) > 0$ 
   such that for all  $t\geq T$, 
   the solution $v$ of problem
    \eqref{veq11}-\eqref{veq22} satisfies
$$
\norm{v(\sigma,\tau-t,\theta_{-\tau}\omega,v_{\tau-t})}^2
 \leq M,
$$
$$
 \int_{\tau-t}^\sigma e^{\frac{5}{4}\lambda (s-\sigma) 
 - 2\alpha \int_{\sigma}^{s} z(\theta_{r-\tau} \omega)dr}
  e^{-2\alpha z(\theta_{s-\tau} \omega)} 
  \norm{ u(s, \tau-t, \theta_{-\tau} \omega, u_{\tau-t})}_{W^{1,p}}^p ds
   \leq M,
$$
$$
 \int_{\tau-t}^\sigma e^{\frac{5}{4}\lambda (s-\sigma) 
 - 2\alpha \int_{\sigma}^{s} z(\theta_{r-\tau} \omega)dr}
  e^{-2\alpha z(\theta_{s-\tau} \omega)} 
$$
$$ 
\times ( \norm{u(s, \tau-t, \theta_{-\tau} \omega, u_{\tau-t})}^2
  + \norm{u(s, \tau-t, \theta_{-\tau} \omega, u_{\tau-t})}_q^q )
    ds 
  \leq M,
$$
where $e^{\alpha z(\theta_{-t}\omega)}
 v_{\tau-t} \in D(\tau-t, \theta_{-t}\omega)$ 
and $M$ is given by
\begin{equation}
\label{est11am}
M = c + c \int_{-\infty}^{\sigma-\tau} 
e^{\frac{5}{4}\lambda (\tau-\sigma+s) 
- 2\alpha \int_{\sigma-\tau}^s z(\theta_{r}\omega)dr } 
e^{-2\alpha z(\theta_{s}\omega)}
$$
$$
\times\left(\norm{\psi_1(s+\tau,\cdot)}_1
+ \norm{g(s+\tau, \cdot)}^2 \right)ds
\end{equation}
for some positive constant $c$ depending
 only on $\lambda$, $q$, and $\gamma$.
\end{lemma}

\begin{proof}
Multiplying \eqref{veq11} through by $v$ 
and integrating over $\mathbb{R}^n$ we obtain
\be\label{v_est_011}
\frac{d}{dt} \norm{v}^2 
+ 2 e^{\alpha (p-2) \zto} \norm{\nabla v}_p^p
 + \left(\frac{7}{4}\lambda - 2 \alpha \zto\right)\norm{v}^2
  + 2\gamma e^{\alpha (q-2)\zto }\norm{v}_q^q
\ee
$$
\leq 2 e^{-2\alpha \zto} \norm{\psi_1(t,\cdot)}_1 
+ \frac{4}{\lambda} e^{-2\alpha \zto} \norm{g(t,\cdot)}^2,
$$
which along with   \eqref{uv1}  yields
\begin{equation*}
\frac{d}{dt}\norm{v}^2 
+ e^{-2\alpha\zto} \left(\frac{1}{2}\lambda\norm{u}^2
 + 2\gamma \norm{u}_q^q + 2\norm{\nabla u}_p^p\right) 
 + \left(\frac{5}{4}\lambda - 2\alpha\zto\right)\norm{v}^2
$$
$$
\leq
 e^{-2\alpha\zto}\left(2\norm{\psi_1(t,\cdot)}_1 
 + \frac{4}{\lambda}\norm{g(t,\cdot)}^2\right).
\end{equation*}
Multiply the above 
 by   $e^{\frac{5}{4}\lambda t 
 - 2\alpha \int_0^t \zro}$ and integrate 
 over $[\tau-t, \sigma]$,
  then replace $\omega$ by $\theta_{-\tau}\omega$
   in order to obtain
\begin{equation*}
\norm{v(\sigma,\tau-t,\theta_{-\tau}\omega,v_{\tau-t})}^2 
$$
$$ 
+ c_1 \int_{\tau-t}^\sigma e^{\frac{5}{4}\lambda (s-\sigma) 
- 2\alpha \int_{\sigma}^{s} z(\theta_{r-\tau}\omega) dr}
 e^{-2\alpha z(\theta_{s-\tau}\omega)} \left(\norm{u}^2 
 + \norm{u}_q^q + \norm{\nabla u}_p^p\right)ds
$$
$$
\leq e^{\frac{5}{4}\lambda (\tau-\sigma-t) 
- 2\alpha \int_{\sigma-\tau}^{-t} z(\theta_{r}\omega) dr} 
\norm{v_{\tau-t}}^2
$$
$$
 + \int_{\tau-t}^\sigma e^{\frac{5}{4}\lambda (s-\sigma) 
 - 2\alpha \int_{\sigma}^{s} z(\theta_{r-\tau}\omega) dr}
  e^{-2\alpha z(\theta_{s-\tau}\omega)}
  \left(2 \norm{\psi_1(s,\cdot)}_1
+ \frac{4}{\lambda}\norm{g(s, \cdot)}^2 \right)ds,
\end{equation*}
where $c_1 = \min\{ {\frac 12} \lambda, 2\gamma, 2\}$.
 Since $e^{\alpha z(\theta_{-t} \omega)} v_{\tau-t} \in D(\tau - t, \theta_{-t}\omega)$ we
 further get  
\begin{equation}\label{esq1}
\norm{v(\sigma,\tau-t,\theta_{-\tau}\omega,v_{\tau-t})}^2 
$$
$$ 
+ c_1 \int_{\tau-t}^\sigma e^{\frac{5}{4}
\lambda (s-\sigma) - 2\alpha \int_{\sigma}^{s}
 z(\theta_{r-\tau}\omega) dr} 
 e^{-2\alpha z(\theta_{s-\tau}\omega)}
  \left(\norm{u}^2 + \norm{u}_q^q + \norm{\nabla u}_p^p\right)ds
$$
$$
\leq e^{\frac{5}{4}\lambda (\tau-\sigma-t)
 - 2\alpha \int_{\sigma-\tau}^{-t} z(\theta_{r}\omega) dr}
  e^{-2\alpha z(\theta_{-t}\omega)} \norm{D(\tau-t,\theta_{-t}\omega)}^2
$$
$$
 + c_2  \int_{-\infty}^{\sigma-\tau} 
 e^{\frac{5}{4}\lambda (\tau-\sigma+s) 
 - 2\alpha \int_{\sigma-\tau}^s z(\theta_{r}\omega)dr } 
 e^{-2\alpha z(\theta_{s}\omega)}\left(\norm{\psi_1(s+\tau,\cdot)}_1
+ \norm{g(s+\tau, \cdot)}^2 \right)ds.
\end{equation}
Note that the last integral in \eqref{esq1} exists because of
\eqref{zergodic} and 
\eqref{g11}.
Since $D \in \cald$  and $\sigma \in [\tau -1, \tau]$, 
it follows from  \eqref{Dom1} that
  there exists  $T = T(\tau, \omega, D, \alpha)>0$
 such that for all  $t\ge T$, 
\begin{equation*}
 e^{\frac{5}{4}\lambda (\tau-\sigma-t)
 - 2\alpha \int_{\sigma-\tau}^{-t} z(\theta_{r}\omega) dr}
  e^{-2\alpha z(\theta_{-t}\omega)} \norm{D(\tau-t,\theta_{-t}\omega)}^2
     \le 1.
\end{equation*}
 which along with \eqref{esq1} and \eqref{p_inequality} completes the proof.
\end{proof}

As a special case of 
  Lemma \ref{est11} we obtain the existence of $\cald$-pullback absorbing
  sets for $\Phi$.

  \begin{lemma}
  \label{abs_alpha1}
   If  \eqref{f11}-\eqref{f31}
    and \eqref{g11} hold,  then  for every
    $\alpha > 0$,  
    the stochastic equation \eqref{seq11}
    with \eqref{seq22}  has   a closed
    measurable $\cald$-pullback absorbing set $K_\alpha=\{ K_\alpha (\tau, \omega):
    \tau \in \R, \omega \in \Omega \} \in \cald$ which  is given by
    \be\label{lem41_1_a}
    K_\alpha  (\tau, \omega) = \{ u \in \ltwo: \| u\|^2 \le
    e^{2\alpha z(\omega)}R(\alpha,\tau,\omega)
    \},
    \ee
     where $R(\alpha,\tau,\omega)$ is  given by
    \begin{equation}
    \label{Rdef1}
R(\alpha,\tau,\omega) = c + c \int_{-\infty}^{0} 
e^{\frac{5}{4}\lambda s - 2\alpha \int_{0}^s z(\theta_{r}\omega)dr }
 e^{-2\alpha z(\theta_{s}\omega)}
$$
$$\times\left(\norm{\psi_1(s+\tau,\cdot)}_1
+ \norm{g(s+\tau, \cdot)}^2 \right)ds
\end{equation}
for some positive constant $c$ depending only 
on $\lambda$, $q$, and $\gamma$.
  \end{lemma}
  
  \begin{proof}
  It follows from Lemma \ref{est11} with $\sigma =\tau$
  that
  there is   $T = T(\tau, \omega, D, \alpha) > 0$
   such that for all  $t\geq T$, the solution $v$ of  
    \eqref{veq11}-\eqref{veq22} satisfies
$$
\norm{v(\tau,\tau-t,\theta_{-\tau}\omega,v_{\tau-t})}^2 
 \leq R(\alpha,\tau, \omega)
$$
for all $e^{\alpha z(\theta_{-t}\omega)}
   v_{\tau-t} \in D(\tau-t, \theta_{-t}\omega)$, which along with
      \eqref{uv1}  indicates that  for all $t \ge T$ and   $u_{\tau -t}
   \in D(\tau-t, \theta_{-t} \omega)$,
    $$
    \norm{u(\tau, \tau -t, \theta_{-\tau} \omega, u_{\tau -t})}
    ^2
    \le e^{2\alpha z(\omega)}R(\alpha,\tau,\omega).
    $$
    Therefore,  by 
    \eqref{pcycle2} and \eqref{lem41_1_a}  we get,  
       for all $t \ge T$,
    \be\label{leabsapha1_p1}
    \Phi (t, \tau -t, \theta_{-t} \omega,
    D(\tau-t, \theta_{-t} \omega) )
    \subseteq K_\alpha (\tau, \omega).
    \ee
  By \eqref{Rdef1}  we have
\begin{equation*}
R(\alpha,\tau-t,\theta_{-t}\omega) = c 
$$
$$ 
+ c \int_{-\infty}^{0} e^{\frac{5}{4}\lambda s
 - 2\alpha \int_{0}^s z(\theta_{r-t}\omega)dr } 
 e^{-2\alpha z(\theta_{s-t}\omega)}
 \left(\norm{\psi_1(s+\tau-t,\cdot)}_1
+ \norm{g(s+\tau-t, \cdot)}^2 \right)ds
$$
$$
= c + c \int_{-\infty}^{-t} e^{\frac{5}{4}
\lambda (s+t) - 2\alpha \int_{-t}^s z(\theta_{r}\omega)dr }
 e^{-2\alpha z(\theta_{s}\omega)}
 \left(\norm{\psi_1(s+\tau,\cdot)}_1
+ \norm{g(s+\tau, \cdot)}^2 \right)ds.
\end{equation*}
This implies 
\begin{equation*}
e^{-\frac{5}{4}t - 2\alpha \int_{0}^{-t} z(\theta_r \omega) dr -2\alpha z(\theta_{-t} \omega )}
\| K_\alpha (\tau -t, \theta_{-t} \omega ) \|^2
$$
$$
=
e^{-\frac{5}{4}t - 2\alpha \int_{0}^{-t} z(\theta_r \omega) dr }
 R(\alpha,\tau-t,\theta_{-t}\omega) 
 = c e^{-\frac{5}{4}t - 2\alpha \int_{0}^{-t} z(\theta_r \omega) dr} 
 $$
 $$
+ c \int_{-\infty}^{-t} e^{\frac{5}{4}\lambda s
 - 2\alpha \int_{0}^s z(\theta_{r}\omega)dr }
  e^{-2\alpha z(\theta_{s}\omega)}
  \left(\norm{\psi_1(s+\tau,\cdot)}_1
+ \norm{g(s+\tau, \cdot)}^2 \right)ds
\end{equation*}
which along with  \eqref{zergodic}
and the convergence of 
  the integrals in \eqref{Rdef1} yields
  \be 
  \label{leabsapha1_p2}
\lim_{t \to \infty}
e^{-\frac{5}{4}t - 2\alpha \int_{0}^{-t} z(\theta_r \omega) dr -2\alpha z(\theta_{-t} \omega )} 
\| K_\alpha (\tau -t, \theta_{-t} \omega ) \|^2
=0.
\ee
By \eqref{leabsapha1_p1}-\eqref{leabsapha1_p2} we find   that
$K_\alpha \in \cald$ is a $\cald$-pullback absorbing set of $\Phi$.
\end{proof}

We will need the following uniform estimates 
on the tails of solutions to \eqref{veq11} in order to
establish the asymptotic compactness
of solutions in
$\ltwo$.

\begin{lemma}\label{est21}
If  \eqref{f11}-\eqref{f31} and \eqref{g11} hold,
then  for  every $\eta > 0$,
 $\tau \in \mathbb{R}$,
 $\sigma \in [\tau-1,\tau]$, 
 $\omega \in \Omega$, 
 $\alpha > 0$  and  $D \in \cald$, 
 there exists  $T = T(\tau, \omega, D, \alpha,  \eta) > 0$
  and $K = K(\tau, \omega, \alpha, \eta) \geq 1$ such that for all $t\geq T$, 
\begin{equation*}
\int_{|x|\geq K} |v(\sigma, \tau-t, \theta_{-\tau}\omega, v_{\tau-t})|^2 dx \leq \eta
\end{equation*}
where $e^{\alpha z(\theta_{-t}\omega)} v_{\tau-t} \in D(\tau-t, \theta_{-t} \omega)$.
\end{lemma}

\begin{proof}
Let $\rho:\mathbb{R}^+ \rightarrow [0, 1]$ be a smooth function which satisfies
\begin{equation*}
\rho(s) = \begin{cases} 0 & 0\leq s \leq 1, \\
                        1 & s\geq 2.
                        \end{cases}
\end{equation*}
 Multiplying \eqref{veq11} through 
 	by $\pxk v$ and integrating over $x \in \mathbb{R}^n$ 
 	we obtain
\be\label{tail_est01}
\frac{1}{2}\frac{d}{dt} 
\inp  v^2 dx - e^{\alpha (p-2)\zto}
 \inp v \text{ div}(|\nabla v|^{p-2} \nabla v) dx 
$$
$$ 
 + (\lambda - \alpha \zto)\inp v^2 dx
$$
$$
= e^{-2 \alpha \zto} \inp f(t,x,u) u dx + e^{-\alpha \zto} \inp g(t, x) v dx.
\ee
First, we have  
\be\label{div_est11}
\inp v \text{ div}(|\nabla v|^{p-2} \nabla v) dx 
$$
$$
\leq -\inp |\nabla v|^p dx 
- \frac{2}{k^2}\int_{k\leq |x| \leq \sqrt{2} k} |v| \rho'(\frac{|x|^2}{k^2})
   |\nabla v|^{p-2} |( x, \nabla v) | dx
$$
$$
\leq \frac{c}{k}
 \int_{k\leq |x| \leq \sqrt{2} k} |v|  |\nabla v|^{p-1} dx
\leq\frac{c_1}{k}\left( \norm{v}^p_p 
+ \norm{\nabla v}^p_p\right).
\ee
 Further, by Young's inequality, we get
\be\label{g_est21}
\inp e^{-\alpha \zto} |g v| dx 
$$
$$
\leq \frac{3\lambda}{8} \inp |v|^2 dx 
+ \frac{2}{3\lambda} e^{-2\alpha \zto} \inp g^2(t,x) dx.
\ee
Hence, using \eqref{f11} and \eqref{div_est11}-\eqref{g_est21}, 
we obtain from \eqref{tail_est01}  that 
\begin{equation*}
\frac{d}{dt}\inp  v^2 dx + 2\gamma e^{-2\alpha \zto}
 \inp |u|^q dx 
$$
$$
+\left(\frac{5}{4}\lambda 
 - 2 \alpha \zto\right)\inp v^2 dx 
\leq \frac{2c_1}{k} e^{-2\alpha \zto}
\left(\norm{u}_p^p + \norm{\nabla u}_p^p\right) 
$$
$$ 
+ 2 e^{-2\alpha \zto} \left(\inp \psi_1(t,x) dx 
+ \frac{2}{3\lambda} \inp g^2(t,x) dx \right).
\end{equation*}
Multiply  the above by
 $e^{\frac{5}{4}\lambda t - 2\alpha \int_0^t z(\theta_{r}\omega)dr}$
  and integrate over $[\tau-t, \sigma]$,
  then replace 
  $\omega$ by $\theta_{-\tau} \omega$
   in order to obtain  
\begin{equation*}
\inp |v(\sigma ,\tau-t,\theta_{-\tau}\omega,v_{\tau-t})|^2 dx 
$$
$$
- e^{\frac{5}{4}\lambda (\tau-\sigma-t)
 - 2\alpha \int_{\sigma-\tau}^{-t} \zro} \inp |v_{\tau-t}|^2 dx
$$
$$
\leq \frac{2 c_1}{k} \int_{\tau-t}^\sigma 
e^{\frac{5}{4}\lambda (s-\sigma) 
- 2\alpha \int_{\sigma}^{s} z(\theta_{r-\tau}\omega) dr}
 e^{-2\alpha z(\theta_{s-\tau}\omega)}
 \norm{u(s,\tau-t,\theta_{-\tau}\omega,v_{\tau-t})}_{W^{1,p}}^p ds
$$
$$
+ 2 \int_{\tau-t}^\sigma e^{\frac{5}{4}\lambda (s-\sigma) 
- 2\alpha \int_{\sigma}^{s} z(\theta_{r-\tau}\omega) dr}
 e^{-2\alpha z(\theta_{s-\tau}\omega)} 
$$
$$ 
\times\left(\inp \psi_1(s,x) dx 
  + \frac{2}{3\lambda} \inp g^2(s,x) dx \right)ds,
\end{equation*}
which along with Lemma    \ref{est11}  implies that
there exists $T_1 =T_1(\tau, \omega, D, \alpha) >0$ such that
for all $t \ge T_1$, 
\be\label{pv_Est}
\inp |v(\sigma,\tau-t,\theta_{-\tau}\omega,v_{\tau-t})|^2 dx 
$$
$$
\leq e^{\frac{5}{4}\lambda (\tau-\sigma-t) 
- 2\alpha \int_{\sigma-\tau}^{-t} \zro}
 \inp |v_{\tau-t}|^2 dx + \frac{2c_1}{k}M
$$
$$
+ 2 \int_{-\infty}^{\sigma-\tau} 
e^{\frac{5}{4}\lambda (\tau-\sigma+s) 
- 2\alpha \int_{\sigma-\tau}^{s} z(\theta_{r}\omega) dr} 
e^{-2\alpha z(\theta_{s}\omega)}
$$
$$ 
\times \left(\int_{|x|\geq k} \psi_1(s+\tau,x) dx 
 + \frac{2}{3\lambda} \int_{|x|\geq k} g^2(s+\tau,x) dx \right)ds,
\ee
where $M$ is the number given by \eqref{est11am}.
 Since $e^{\alpha z(\theta_{-t}\omega)}
 v_{\tau-t}\in D(\tau-t,\theta_{-t}\omega)$, 
 by \eqref{Dom1} there is 
  $T_2 = T_2(\tau,\omega,D, \alpha, \eta)   \ge T_1$
   such that for all $t\geq T_2$,
\begin{equation}\label{vtt_Est}
e^{\frac{5}{4}\lambda (\tau-\sigma-t) 
- 2\alpha \int_{\sigma-\tau}^{-t} \zro} \inp |v_{\tau-t}|^2 dx 
$$
$$
\leq e^{-\frac{5}{4}\lambda t 
- 2\alpha \int_{0}^{-t} z(\theta_{r}\omega) dr} 
e^{-2\alpha z(\theta_{-t}\omega)}
 \norm{D(\tau-t,\theta_{-t}\omega)}^2 < \eta.
\end{equation}
 Since the integral in \eqref{est11am} exists, we find that
there exists $K=K(\tau, \omega, \alpha, \eta) \ge 1$ such  that
both
  $ \frac{2c_1}{k}M$ and the last integral in \eqref{pv_Est}
  are bounded by $\eta$ for all $k \ge K$.
  This   together with \eqref{pv_Est} and 
  \eqref{vtt_Est} implies that for all $t\geq T_2$ and $k \geq K$,
\begin{equation*}
\int_{|x|\geq \sqrt{2} k}
 |v(\sigma,\tau-t, \theta_{-\tau}\omega,v_{\tau-t})|^2 dx 
 \leq \inp |v(\sigma,\tau-t, \theta_{-\tau}\omega,v_{\tau-t})|^2 dx < 3\eta,
\end{equation*}
as desired.  
\end{proof}

We now prove the  
 asymptotic compactness of solutions of equation  \eqref{veq11}
in $\ltwo$. 

\begin{lemma}
\label{asyv1}
If  \eqref{f11}-\eqref{f31}  and \eqref{g11} hold, 
then  for every $\alpha >0$,
 $\tau \in \R$, $\omega \in \Omega$   and $D  \in \cald$,
 the sequence
 $v(\tau, \tau -t_n,  \theta_{-\tau} \omega,   v_{0,n}  ) $   has a
   convergent
subsequence in $\ltwo $ provided
  $t_n \to \infty$ and 
 $ v_{0,n}e^{\alpha z(\theta_{-t_n}\omega)}  \in D(\tau -t_n, \theta_{ -t_n} \omega )$.
 \end{lemma}
 
 \begin{proof}
 Let $\widetilde{v}_n = v(\tau -1,  \tau -t_n, \theta_{-\tau} \omega, v_{0,n} )$.
 It follows from 
  Lemma \ref{est11}  that   
  $\{\widetilde{v}_n\}_{n=1}^\infty$ is bounded in $\ltwo$.
  Therefore, by   Corollary  \ref{comv1},  
   there exist $r  \in (\tau -1, \tau)$
   and $ \xi  \in \ltwo$ such that,  up to a subsequence,
    $$
    v(r,    \tau -t_n, \theta_{-\tau} \omega, v_{0,n} )
    = 
    v(r,  \tau - 1, \theta_{-\tau} \omega,
    \widetilde{v} _n )
    \to  \xi 
    \quad \text{ in } L^2(\o_k)
    \text{ for  all }  k \in \N,
   $$
   which along with Lemma \ref{est21} shows that
     \be
    \label{pav_21}
    v(r,    \tau -t_n, \theta_{-\tau} \omega, v_{0,n} )
     \to  \xi 
    \quad \text{ in } \ltwo.
  \ee
  On the other hand, 
  by continuity in initial data and 
  the fact  $r \in (\tau-1, \tau)$  we get 
    $$
    \norm{ v(\tau,  r, \theta_{-\tau} \omega,
    v(r,  \tau -t_n, \theta_{-\tau} \omega, v_{0,n} ) ) 
    - v(\tau, r,  \theta_{-\tau} \omega, \xi)} 
    $$\be\label{pav_81}\le
    c  \norm{  
   v(r,  \tau -t_n, \theta_{-\tau} \omega, v_{0,n} )  - \xi  
    }.
   \ee
    Hence, by \eqref{pav_21} and \eqref{pav_81} we obtain 
    $$
   v(\tau,  \tau -t_n, \theta_{-\tau} \omega, v_{0,n} )
   $$
$$   =    v(\tau,  r, \theta_{-\tau} \omega,
    v(r,  \tau -t_n, \theta_{-\tau} \omega, v_{0,n} ) ) 
    \to  v(\tau, r,  \theta_{-\tau} \omega, \xi)
    \ \ \mbox{in } \ \ltwo
    $$
    which completes    the proof.
   \end{proof}

    We now present the existence and  uniqueness of $\cald$-pullback attractors
    of equation   \eqref{seq11} in $\ltwo$.

    \begin{theorem}
    \label{eatta1}
    If   \eqref{f11}-\eqref{f31}  and \eqref{g11} hold,
    then  for every $\alpha > 0$,
   system \eqref{seq11}-\eqref{seq22}  
   possesses  a unique $\cald$-pullback attractor $\cala
   =\{\cala(\tau, \omega):
      \tau \in \R, \ \omega \in \Omega \} \in \cald$
  in $\ltwo$.    Moreover,   if there is $T>0$  such that
 $f(\cdot, x,s)$, $g(\cdot, x)$ and $\psi_1(\cdot,x)$ are 
 periodic in their first argument with period $T$
 for  every  fixed $x \in \R^n$
 and $s\in \R$,   then the attractor  
 is also $T$-periodic; that is,
 $\cala(\tau +T, \omega) =\cala(\tau, \omega)$
 for all $\tau \in \R$  and $\omega \in \Omega$.
 \end{theorem}

\begin{proof}
Since    $\Phi$ has a closed measurable 
absorbing set $K_\alpha$ given by \eqref{lem41_1_a}
and is    $\cald$-pullback
asymptotically  compact 
by \eqref{uv1}  and Lemma \ref{asyv1},
the  existence  and  uniqueness of
$\cald$-pullback attractor $\cala$  follows from
\cite{wan5, wan7} immediately. 
If $f$, $g$  and $\psi_1$ are $T$-periodic in time,
then both $\Phi$  and $K_\alpha$ are also $T$-periodic;
that is,
$\Phi (t, \tau +T, \omega,  \cdot )
= \Phi (t, \tau,  \omega,  \cdot )$
and   $ 
K_\alpha (\tau +T, \omega)
=K_\alpha (\tau, \omega )$
 for all  $t \in \R^+$, $\tau \in \R$ and
$\omega \in \Omega$.
 As a consequence,   we obtain the  periodicity of $\cala$
 from \cite{wan5}. 
 \end{proof}

\section{Upper-Semicontinuity of Random Attractors}
In this section, we prove convergence of random attractors of \eqref{seq11} as the intensity $\alpha$ of noise
approaches zero. From now on, we write the solution of \eqref{seq11}-\eqref{seq22} as $u_\alpha$ and the corresponding
cocycle as $\phi_\alpha$. Given $\tau \in \R$ and $\omega \in \Omega$, let

\be\label{us_R}
\tilde{R}(\tau,\omega) = c + c\int_{-\infty}^{0}e^{\frac{5}{4}\lambda s + 2 \abs{\int_0^{s}z(\theta_r \omega)dr} + 2\abs{z(\theta_s \omega)}}(\norm{\psi_1(s+\tau,\cdot)}_1 + \norm{g(s+\tau,\cdot}^2)ds
\ee
and

\be\label{BDef}
B(\tau, \omega) = \{u \in \ltwo : \norm{u}^2 \leq e^{2\abs{z(\theta_r \omega)}\tilde{R}(\tau,\omega)}\}
\ee
where the positive number $c$ is as in \eqref{Rdef1}. Then we have,

$$
\norm{B(\tau-t, \theta_{-t}\omega)}^2 \leq e^{-2\abs{z(\theta_{-t} \omega)}}
$$
$$
+ ce^{2\abs{z(\theta_{-t} \omega)}}\int_{-\infty}^0 e^{\frac{9}{8}\lambda s + 2\abs{\int_{0}^{s} z(\theta_{r-t} \omega)dr}+2\abs{z(\theta_{s-t} \omega)}}
$$
$$
\times(\norm{\psi_1(s+\tau-t,\cdot)}_1 + \norm{g(s+\tau-t,\cdot}^2)ds
$$
$$
\leq ce^{2\abs{z(\theta_{-t} \omega)}}( 1 + e^{\frac{9}{8}\lambda t}\int_{-\infty}^t e^{\frac{9}{8}\lambda s + 2\abs{\int_{-t}^{s} z(\theta_{r} \omega)dr}+2\abs{z(\theta_{s} \omega)}}
$$
$$
\times(\norm{\psi_1(s+\tau,\cdot)}_1 + \norm{g(s+\tau,\cdot}^2)ds).
$$
Therefore, for all $\alpha \in (0,1]$ and $t \geq 0$, we obtain

$$
e^{-\frac{5}{4}\lambda t - 2\alpha \int_{0}^{-t} z(\theta_{r} \omega)dr-2\alpha z(\theta_{-t} \omega)}\norm{B(\tau-t, \theta_{-t}\omega)}^2 \leq ce^{-\frac{5}{4}\lambda t + 2\abs{ \int_{-t}^{0} z(\theta_{r} \omega)dr}+4\abs{ z(\theta_{-t} \omega)}} 
$$
$$
+ ce^{-\frac{1}{8}\lambda t + 4\abs{ \int_{-t}^{0} z(\theta_{r} \omega)dr}+4\abs{ z(\theta_{-t} \omega)}}\int_{-\infty}^{-t}e^{\frac{9}{8}\lambda s + 2\abs{ \int_{s}^{0} z(\theta_{r} \omega)dr}+2\abs{ z(\theta_{s} \omega)}}
$$
$$
\times(\norm{\psi_1(s+\tau)}_1 + \norm{g(s+\tau)}^2)ds.
$$
Note that the last integral exists by \eqref{g11}. Thus, for every $\eta > 0$, $\tau \in \R$, and $\omega \in \Omega$, there exists a $T = T(\tau, \omega, \eta)$ such that for all $t \geq T$ and for all $\alpha \leq 1$,

\be\label{above_B}
e^{-\frac{5}{4}\lambda t - 2\alpha \int_{0}^{-t} z(\theta_{r} \omega)dr-2\alpha z(\theta_{-t} \omega)}\norm{B(\tau-t, \theta_{-t}\omega)}^2 \leq \eta.
\ee
Based on \eqref{above_B}, we are able to prove the following uniform estimates on the tails of functions in
random attractors.

\begin{lemma}\label{usc_tails}
If \eqref{f11}-\eqref{f31}  and \eqref{g11} hold, then for every $\eta > 0$, $\tau \in \R$ and $\omega \in \Omega$, there exists
$K = K(\tau,\omega,\eta) \geq 1$ such that for all $k \geq K$,
$$
\int_{\abs{x}\geq k}\abs{\xi(x)}^2dx \leq \eta \text{ for all } \xi \in \bigcup_{0 < \alpha \leq 1}\cala_{\alpha}(\tau,\omega).
$$

\end{lemma}

\begin{proof}
Following the proof of Lemmas \ref{est11} and \ref{est21}, by \eqref{above_B} one can verify that for every $\eta > 0$, $\tau \in \R$, and $\omega \in \Omega$, there exists a $T = T(\tau, \omega, \eta)$ and $K = K(\tau,\omega,\eta) \geq 1$, such that for all $t \geq T$ and $k \geq K$ and for all $\alpha \in (0,1]$, the solution $u_\alpha$ of \eqref{seq11}-\eqref{seq22} satisfies,

\be\label{above_int}
\int_{\abs{x}\geq k}\abs{u_\alpha(\tau,\tau-t,\theta_{-\tau}\omega,u_{\alpha,\tau-t})}^2dx \leq \eta
\ee
where $u_{\alpha,\tau-t} \in B(\tau - t,\theta_{-\tau}\omega)$ with $B$ given by \eqref{BDef}. By  \eqref{Rdef1}, \eqref{lem41_1_a} and \eqref{us_R}-\eqref{BDef} we have

\be\label{Union_Attractors}
\bigcup_{0< \alpha \leq 1}\cala_\alpha(\tau,\omega) \subseteq \bigcup_{0< \alpha \leq 1}K_\alpha(\tau,\omega) \subseteq B(\tau,\omega).
\ee
Let $\xi \in\cala_\alpha(\tau,\omega)$ for some $\alpha \in (0,1]$. By the invariance of $\cala_\alpha$, there exists $\zeta \in \cala_\alpha(\tau - T,\theta_{-T} \omega)$
such that $\xi = u_\alpha(\tau,\tau-T,\theta_{-\tau}\omega,\zeta)$, which along with \eqref{above_int}-\eqref{Union_Attractors} implies that for all $k \geq K$,

\begin{equation*}
\int_{\abs{x}\geq k}\abs{\xi(x)}^2dx = \int_{\abs{x}\geq k}\abs{u_\alpha(\tau,\tau-T,\theta_{-\tau}\omega,\zeta)}^2dx \leq \eta,
\end{equation*}
as desired.
\end{proof}

In the limiting case $\alpha = 0$, the stochastic equation \eqref{seq11} reduces to a deterministic one:

\be\label{deterministic}
  {\frac {\partial u}{\partial t}}  + \lambda u
  - {\rm div} \left (|\nabla u |^{p-2} \nabla u \right )
  =f(t,x,u ) 
  +g(t,x) 
\quad 
\text{with}
\quad 
u(\tau,x)=u_\tau(x)
\ee
for $x \in \R^n$. Denote the cocycle of \eqref{deterministic} in $\ltwo$ by $\phi_0$. As in \eqref{Dom21}, let $\cald_0$ be the collection of
subsets of $\ltwo$ given by,

 \begin{equation*}
 \cald  = \{ 
   D =\{ D(\tau) \subseteq \ltwo:
    \tau \in \R\}: \lim_{s \to  - \infty} e^{ \frac{5}{4}\lambda s}
  \| D( \tau + s) \|^2 =0, \forall \tau \in \R\}.
 \end{equation*}
Note that Theorem \ref{eatta1} is also valid when $\alpha = 0$; more precisely, $\phi_0$ has a unique $\cald_0$-pullback
attractor $\cala_0 = \{\cala_0(\tau):\tau \in \R \}\in \cald_0$ and has a $\cald_0$-pullback absorbing set $K_0 = \{K_0(\tau):\tau \in \R\}$
where $K_0(\tau)$ is defined by

\be\label{d_K}
K_0(\tau) = \{u \in \ltwo : \norm{u}^2 \leq R_0(\tau) \}
\ee
with

\be\label{d_R}
R_0(\tau) = c + c\int_{-\infty}^0 e^{\frac{5}{4}\lambda s}(\norm{\psi_1(s+\tau)}_1 + \norm{g(s+\tau)}^2)ds.
\ee
The constant $c$ in \eqref{d_R} is the same as in \eqref{Rdef1}. It follows from \eqref{Rdef1}, \eqref{lem41_1_a} and \eqref{d_K}-\eqref{d_R} that

\be\label{d_lim_K}
\limsup_{\alpha \to 0}\norm{K_\alpha(\tau,\omega)} = \norm{K_0(\tau)}.
\ee
In the sequel, we further assume there exists $\psi_5 \in L^{\infty}(\R,L^{\infty}(\R^n))$ such that for all $t,s \in \R$ and
$x\in \R^n$,

\be
\label{f71}
\abs{\frac {\partial f}{\partial s} (t, x, s)}   \le \psi_5 (t,x)(1+\abs{s}^{q-2}).
\ee
Under condition \eqref{f71}, we have the following relations between solutions of \eqref{seq11} and \eqref{deterministic}.

\begin{lemma}\label{usc_conv}
Suppose \eqref{f11}-\eqref{f31}, \eqref{f71} and \eqref{g11} hold. If $u_\alpha$ and $u$ are the solutions of \eqref{seq11}-\eqref{seq22} and \eqref{deterministic} with initial data $u_{\alpha,\tau}$ and $u_\tau$, respectively, then for every $\tau \in \R$, $\omega \in \Omega$, $T > 0$, and $\epsilon \in (0,1)$, there exists a $\alpha_0=\alpha_0(\tau,\omega,T,\epsilon) \in (0,1)$ such that for all $\alpha \leq \alpha_0$ and $t \in [\tau, \tau+T]$,
$$
\norm{u_\alpha(t,\tau,\omega,v_{\alpha,\tau}) - u(t,\tau,u_{\tau}}^2 \leq c\norm{v_{\alpha,\tau}-u_\tau}^2 + c\epsilon (1 + \norm{v_{\alpha,\tau}}^2 + \norm{u_{\tau}}^2),
$$
where $c$ and $c_0$ are positive constants independent of $\alpha$ and $\epsilon$.
\end{lemma}
\begin{proof}
Let $\xi = v_\alpha - u$. Then by subtracting \eqref{veq11} from \eqref{deterministic} and taking the inner product with $\xi$, we get
\be\label{usc_1}
\frac{1}{2}\frac{d}{dt} \norm{\xi}^2 + \int_{\R^n}(e^{\alpha(p-2)z(\theta_t\omega)}\abs{\nabla v_\alpha}^{p-2}\nabla v_\alpha-\abs{\nabla u}^{p-2}\nabla u)\cdot \nabla \xi dx + \lambda \norm{\xi}^2 
\ee
$$
= \alpha z(\theta_t\omega)\norm{\xi}^2 + \alpha z(\theta_t\omega)(u, \xi) +  \int_{\R^n}(e^{-\alpha z(\theta_t\omega)}f(t,x,e^{\alpha z(\theta_t\omega)}v_\alpha)-f(t,x,u))\xi dx 
$$
$$
+ (e^{-\alpha z(\theta_t\omega)}-1)\int_{\R^n}g(t,x)\xi dx.
$$
For the second term on the left hand side of \eqref{usc_1} we have that,
$$
 \int_{\R^n}(e^{\alpha(p-2)z(\theta_t\omega)}\abs{\nabla v_\alpha}^{p-2}\nabla v_\alpha-\abs{\nabla u}^{p-2}\nabla u)\cdot \nabla \xi dx
$$
\be\label{usc_2}
= \int_{\R^n}e^{\alpha(p-2)z(\theta_t\omega)}(\abs{\nabla v_\alpha}^{p-2}\nabla v_\alpha-\abs{\nabla u}^{p-2}\nabla u)\cdot \nabla \xi dx
\ee
$$
+  \int_{\R^n}(e^{\alpha(p-2)z(\theta_t\omega)}-1)\abs{\nabla u}^{p-2}\nabla u)\cdot \nabla \xi dx
$$
By the monotonicity of the p-laplace operator, see e.g. \cite{dibe1}, we have that there is a positive number $\beta$ such that
\be\label{usc_3}
(\abs{\nabla v_\alpha}^{p-2}\nabla v_\alpha-\abs{\nabla u}^{p-2}\nabla u)\cdot ( \nabla v_\alpha - \nabla u) \geq \beta \abs{\nabla v_\alpha - \nabla u}.
\ee
By Young's inequality, we find that for every $\tau \in \R$, $\omega \in \Omega$, $T > 0$, and $\epsilon \in [0,1]$, there exists a $\alpha_1=\alpha_1(\tau,\omega,T,\epsilon) > 0$ such that for all $\alpha \in [0, \alpha_1]$,
\be\label{usc_4}
\abs{\int_{\R^n}(e^{\alpha(p-2)z(\theta_t\omega)}-1)\abs{\nabla u}^{p-2}\nabla u)\cdot \nabla \xi dx }
\ee
$$
\leq \frac{1}{2}\beta e^{\alpha(p-2)z(\theta_t\omega)}\int_{\R^n}\abs{\nabla \xi}^p dx + \epsilon \int_{\R^n}\abs{\nabla u}^p dx.
$$
It follows from \eqref{usc_2}-\eqref{usc_4} that 
\be\label{usc_5}
\int_{\R^n}(e^{\alpha(p-2)z(\theta_t\omega)}\abs{\nabla v_\alpha}^{p-2}\nabla v_\alpha-\abs{\nabla u}^{p-2}\nabla u)\cdot \nabla \xi dx 
\ee
$$
\geq \frac{1}{2}\beta e^{\alpha(p-2)z(\theta_t\omega)}\int_{\R^n}\abs{\nabla \xi}^p dx - \epsilon \int_{\R^n}\abs{\nabla u}^p dx.
$$
For the third term on the right hand side of \eqref{usc_1} by \eqref{f21}, \eqref{f31}, and \eqref{f71}  and Young's inequality, we find that
$$
\int_{\R^n}(e^{-\alpha z(\theta_t\omega}f(t,x,e^{\alpha z(\theta_t\omega)}v_\alpha)-f(t,x,u))\xi dx 
$$
$$
= \int_{\R^n}e^{-\alpha z(\theta_t\omega)}(f(t,x,e^{\alpha z(\theta_t\omega)}v_\alpha)-f(t,x,e^{\alpha z(\theta_t\omega)}u))\xi dx
$$
$$
+ \int_{\R^n}(e^{-\alpha z(\theta_t\omega)}f(t,x,e^{\alpha z(\theta_t\omega)}u)-f(t,x,e^{\alpha z(\theta_t\omega)}u))\xi dx 
$$
$$
+ \int_{\R^n}(f(t,x,e^{\alpha z(\theta_t\omega)}u)-f(t,x,u))\xi dx = \int_{\R^n}\frac{\partial f}{ \partial s}(t,x,s)\xi^2 dx 
$$
$$
+ (e^{-\alpha z(\theta_t\omega)} - 1)\int_{\R^n}(f(t,x,e^{\alpha z(\theta_t\omega)}u))\xi dx+ (e^{\alpha z(\theta_t\omega)} - 1)\int_{\R^n}\frac{\partial f}{ \partial s}(t,x,s)u\xi dx 
$$
$$
\leq \int_{\R^n}\psi_5(t,x)(\xi^2 +\abs{\xi}^q)dx + \abs{e^{-\alpha z(\theta_t\omega)} - 1}\int_{\R^n}
$$
$$
\times (\psi_2(t,x)e^{\alpha (q-1) z(\theta_t\omega)}\abs{u}^{q-1} + \psi_3(t,x))\abs{\xi} dx 
$$
$$
+ \abs{e^{\alpha z(\theta_t\omega)} - 1}\int_{\R^n}\psi_4(t,x)\abs{\xi}\abs{u} dx \leq c_1\norm{\xi}^2
$$
$$
+c_1 \abs{e^{-\alpha z(\theta_t\omega)} - 1}\int_{\R^n}(e^{\alpha (q-1) z(\theta_t\omega)}(\abs{u}^{q}+\abs{v_\alpha}^q + \psi_3(t,x)^{q_1}) dx 
$$
\be\label{usc_6}
+ c_1\abs{e^{\alpha z(\theta_t\omega)} - 1}\norm{u}^2 dx.
\ee
So for every $\tau \in \R$, $\omega \in \Omega$, $T > 0$, and $\epsilon \in [0,1]$, there exists a $\alpha_2=\alpha_2(\tau,\omega,T,\epsilon) > 0$ such that for all $\alpha \in [0, \alpha_2]$ by \eqref{usc_6} we have,
$$
\int_{\R^n}(e^{-\alpha z(\theta_t\omega}f(t,x,e^{\alpha z(\theta_t\omega)}v_\alpha)-f(t,x,u))\xi dx 
$$\be\leq c_2\norm{\xi}^2 + \epsilon c_2(1 + \norm{u}^2 + \norm{v_\alpha}^q_q+\norm{u}^q_q).
\ee
Similarly, for every $\tau \in \R$, $\omega \in \Omega$, $T > 0$, and $\epsilon \in [0,1]$, there exists a $\alpha_3=\alpha_3(\tau,\omega,T,\epsilon) > 0$ such that for all $\alpha \in [0, \alpha_3]$ we have,
\be\label{usc_7}
(e^{-\alpha z(\theta_t\omega)}-1)\int_{\R^n}g(t,x)\xi dx \leq c_3\norm{\xi}^2 + \epsilon c_3 \norm{g(t,\cdot)}^2
\ee
So by \eqref{usc_5}-\eqref{usc_7}, and Young's inequality we have from \eqref{usc_1} that
$$
\frac{d}{dt} \norm{\xi}^2 \leq c_4\norm{\xi}^2 + \epsilon c_4(1 + \norm{u}^2 + \norm{v_\alpha}^q_q+\norm{u}^q_q + \norm{\nabla u}^p_p + \norm{g(t,\cdot)}^2).
$$
Solving this inequality we have,
\begin{multline}\label{usc_8}
\norm{\xi}^2 \leq e^{c_4(t-\tau)}\norm{\xi}^2 \\
+ \epsilon c_4e^{c_4(t-\tau)}\int_{\tau}^{t}(1 + \norm{u(s,\tau,\omega,u_{\tau})}^2 + \norm{v_\alpha}^q_q+\norm{u}^q_q + \norm{\nabla u}^p_p + \norm{g(t,\cdot)}^2)ds.
\end{multline}
By \eqref{v_est_011} we have that for all $\alpha \in [0,1]$,
\be\label{usc_9}
\frac{d}{dt} \norm{v_\alpha}^2 + 2 e^{\alpha (p-2) \zto} \norm{\nabla v_\alpha}_p^p + c_5\norm{v_\alpha}^2 + 2\gamma e^{\alpha (q-2)\zto }\norm{v_\alpha}_q^q
$$
$$
\leq c_6\norm{v_\alpha}^2 + 2 e^{-2\alpha \zto} \norm{\psi_1(t,\cdot)}_1 + \frac{4}{\lambda} e^{-2\alpha \zto} \norm{g(t,\cdot)}^2.
\ee
By \eqref{usc_9} we have that for all $\alpha \in [0,1]$ and $t \in [\tau,\tau+T]$,
\be\label{usc_10}
\norm{v_\alpha}^2 + \int_{\tau}^{t}\norm{\nabla v_\alpha}_p^p + \norm{v_\alpha}^2 + \norm{v_\alpha}_q^qds
$$
$$
\leq e^{c_6(t-\tau)}\norm{v_{\alpha,\tau}}^2 + c_7e^{c_6(t-\tau)}\int_{\tau}^{t}(\norm{\psi_1(s,\cdot)}_1 + \norm{g(s,\cdot)}^2)ds.
\ee
It is clear that \eqref{usc_10} also holds for $\alpha=0$, so we have by \eqref{usc_8} and \eqref{usc_10},
\be\label{usc_11}
\norm{v_\alpha(t,\tau,\omega,v_{\alpha,\tau}) - u(t,\tau,u_{\tau}}^2 \leq e^{c_4(t-\tau)}\norm{v_{\alpha,\tau}-u_\tau}^2 
$$
$$
+ \epsilon c_8e^{c_9(t-\tau)}(1 + \norm{v_{\alpha,\tau}}^2 + \norm{u_{\tau}}^2 + \int_{\tau}^{t}(\norm{\psi_1(s,\cdot)}_1 + \norm{g(s,\cdot)}^2)ds.
\ee
Since
\begin{equation*}
\norm{u_\alpha(t,\tau,\omega,v_{\alpha,\tau}) - v_\alpha(t,\tau,\omega,v_{\alpha,\tau})}^2 = \abs{e^{\alpha z(\theta_t\omega)}-1}\norm{u_\alpha(t,\tau,\omega,v_{\alpha,\tau})},
\end{equation*}
we have the desired result from \eqref{usc_10} and \eqref{usc_11}.

\end{proof}

\begin{lemma}\label{final_usc}
Let \eqref{f11}-\eqref{f31}, \eqref{f71} and \eqref{g11} hold. Given $\tau \in \R$ and $\omega \in \Omega$, if $\alpha_n \to 0$ and $u_n \in \cala_{\alpha_n}(\tau,\omega)$, then $\{u_n \}$ has a convergent subsequence in $\ltwo$.
\end{lemma}

\begin{proof}
It follows from the invariance of $\cala_{\alpha_n}$ and the fact $u_n \in \cala_{\alpha_n}(\tau,\omega)$ that for each $n \in \N$, there
is $\xi_n \in \cala_{\alpha_n}(\tau-1,\theta_{-1}\omega)$ such that

\be\label{def_usc}
u_n = u_{\alpha_n}(\tau,\tau-1,\theta_{-\tau}\omega,\xi_n).
\ee
By \eqref{uv1}, for the solution $v_{\alpha_n}$ of \eqref{veq11}-\eqref{veq22} we get, for $s \geq \tau - 1$,

\begin{multline}\label{tran_usc}
v_{\alpha_n}(s,\tau-1,\theta_{-\tau}\omega,\zeta_n) \\
= e^{-\alpha_n z(\theta_{s-\tau}\omega)}u_{\alpha_n}(s,\tau-1,\theta_{-\tau}\omega,\xi_n)
\quad \text{where}\quad \zeta_n = e^{-\alpha_n z(\theta_{s-\tau}\omega)}\xi_n.
\end{multline}
By \eqref{Union_Attractors} and \eqref{tran_usc} we find that $\{\zeta_n\}$ is a bounded sequence in $\ltwo$, and hence, as in Lemma
\ref{comv11}, there exist $r \in (\tau-1,\tau)$ and $\phi \in \ltwo$ such that, up to a subsequence,

\begin{equation*}
v_{\alpha_n}(r,\tau - 1, \theta_{-\tau}\omega,\zeta_n) \to \phi \quad \text{in} \quad L^2(\o_k) \quad \text{for all} \quad k \in \N,
\end{equation*}
which along with \eqref{tran_usc} yields,

\be\label{usc_tran2}
u_{\alpha_n}(r,\tau - 1, \theta_{-\tau}\omega,\xi_n) \to \phi \quad \text{in} \quad L^2(\o_k) \quad \text{for all} \quad k \in \N.
\ee
Since $\xi_n \in \cala_{\alpha_n}(\tau-1,\theta_{-1}\omega)$, by the invariance of attractor, we have $u_{\alpha_n}(r,\tau-1,\theta_{-\tau}\omega,\xi_n)\in \cala_{\alpha_n}(r,\theta_{r-\tau})$. This together with \eqref{usc_tran2} and Lemma \ref{usc_tails} implies

\be\label{usc_tran3}
u_{\alpha_n}(r,\tau - 1, \theta_{-\tau}\omega,\xi_n) \to \phi \quad \text{in} \quad \ltwo.
\ee
By \eqref{usc_tran3} and Lemma \ref{usc_conv} we get

\begin{equation*}
u_{\alpha_n}(\tau,\tau - 1, \theta_{-\tau}\omega,\xi_n) = u_{\alpha_n}(\tau,r, \theta_{-\tau}\omega,u_{\alpha_n}(r,\tau - 1, \theta_{-\tau}\omega,\xi_n)) \to u(\tau,r, \theta_{-\tau}\omega,\phi)
\end{equation*}
in $\ltwo$, which along with \eqref{def_usc} completes the proof.

\end{proof}

\begin{lemma}
If \eqref{f11}-\eqref{f31}, \eqref{f71} and \eqref{g11} hold, then for every $\tau \in \R$, and $\omega \in \Omega$,

\begin{equation*}
\lim_{\alpha \to 0}\text{dist}_{\ltwo}(\cala_\alpha(\tau,\omega),\cala_0(\tau,\omega))=0
\end{equation*}
\end{lemma}

\begin{proof}
This follow from Theorem 3.7 in \cite{wan6} directly based on \eqref{d_lim_K}, and Lemmas \ref{usc_conv} and \ref{final_usc}
\end{proof}

\chapter{Conclusion}
This thesis looked at two families of $p$-Laplace equations driven by stochastic and non-autonomous noise. Well-posedness of the equations was outlined briefly, allowing us to define a cocycle to model the dynamics. Estimates on the equations allowed us to define absorbing sets for these equations. Estimates on the tails of solutions to these equations were also used to show that these cocycles were pullback-asymptotically compact. These two properties immediately give us the existence and uniqueness of random pullback attractors, defined with respect to certain universes. The existence of these objects tells us that over long periods of time, the dynamics is confined to a compact subset of the phase space, and cannot grow arbitrarily largely.

In the multiplicative noise case, we also showed a structural property of the attractor. Specifically, we showed that the random pullback attractor for the stochastic equation approached that of the deterministic equation as $\alpha \to 0$ in an upper semicontinuous way. Roughly speaking, this means that the random attractor's dynamics was approximately contained within the deterministic one for very small $\alpha$. In other words, perturbing the equation with a small noise term does not cause the attractor to explode or bifurcate to a much larger attractor. Lower semicontinuity, that is, showing that the deterministic attractor is approximately contained in the random one for very small $\alpha$, is a much more difficult question in general, and requires more detailed knowledge of the asymptotic dynamics. The overall interpretation of having proven upper semicontinuity, but not lower semicontinuity, is that a small stochastic perturbation may cause the attractor to collapse or shrink in some sense, but it cannot become a much larger set. 

There are many related open questions that were not addressed in this thesis. One may be able use the monotone properties of the $p$-Laplace operator to bound the attractors in some sense. It would also be interesting to look at other types of attracting sets, such as inertial manifolds or exponential attractors, which give some information about how quickly solutions will be attracted to them. It would also be interesting to explore other types of dynamics, such as forward attractors, although there are some difficulties in studying the forward dynamics of non-autonomous and stochastic equations. There are of course questions of regularity of attractors, and other detailed structural problems associated with them.

%
%
%
%
\begin{References}
%
%
\bibliographystyle{plainnat}
\bibliography{thesis.bib}
%
%
\end{References}
\copyrightpage
\end{document}